\documentclass{amsart}

\usepackage{amsfonts}
\usepackage{amssymb}
\usepackage{enumerate}
\usepackage{amsmath}
\usepackage{amsthm}
\usepackage{graphicx}
\usepackage[english]{babel}

\numberwithin{equation}{section}

\newtheorem{theorem}{Theorem}[section]
\newtheorem{lemma}[theorem]{Lemma}

\newtheorem{question}[theorem]{Question}

\newtheorem{corollary}[theorem]{Corollary}

\newtheorem{fact}[theorem]{Fact}
\newtheorem*{x}{Theorem \ref{t:main}}
\newtheorem*{y}{Corollary \ref{c:main}}
\newtheorem*{z}{Theorem \ref{t:Kg}}

\theoremstyle{definition}

\newtheorem{remark}[theorem]{Remark}
\newtheorem{definition}[theorem]{Definition}
\newtheorem{notation}[theorem]{Notation}

\DeclareMathOperator{\pr}{pr}
\DeclareMathOperator{\diam}{diam}
\DeclareMathOperator{\dist}{dist}

\DeclareMathOperator{\inter}{int}
\DeclareMathOperator{\mesh}{mesh}
\DeclareMathOperator{\cl}{cl}

\DeclareMathOperator{\supp}{supp}
\DeclareMathOperator{\St}{St}

\title[Inductive topological Hausdorff dimensions and fibers]{Inductive topological Hausdorff dimensions and fibers of generic continuous functions}

\author{Rich\'ard Balka}
\address{Department of Mathematics, University of Washington, Box 354350, Seattle, WA 98195-4350, USA}
\email{balka@math.washington.edu}
\address{Alfr\'ed R\'enyi Institute of Mathematics, Hungarian Academy of Sciences, PO Box 127, 1364 Budapest, Hungary}
\email{balka.richard@renyi.mta.hu}
\thanks{Supported by the Hungarian Scientific Research Fund grants no.~72655 and 104178.}



\date{}

\begin{document}

\begin{abstract} In an earlier paper Buczolich, Elekes and the author
introduced a new concept of dimension for metric spaces, the so called topological Hausdorff dimension. They
proved that it is precisely the right notion to describe the Hausdorff dimension of
the level sets of the generic real-valued continuous function (in the sense of Baire
category) defined on a compact metric space $K$.

The goal of this paper is to determine the Hausdorff dimension of the
fibers of the generic continuous function from $K$ to $\mathbb{R}^n$.
In order to do so, we define the $n$th inductive topological
Hausdorff dimension, $\dim_{t^nH} K$.
Let $\dim_H K$, $\dim_t K$ and $C_n(K)$ denote the Hausdorff and topological dimension of $K$ and the Banach
space of the continuous functions from $K$ to $\mathbb{R}^n$. We show that
$\sup_{y\in \mathbb{R}^n} \dim_{H}f^{-1}(y) = \dim_{t^nH} K -n$ for the generic $f \in C_n(K)$,
provided that $\dim_t K\geq n$, otherwise every fiber is finite.

In order to prove the above theorem we give some equivalent definitions for the inductive topological Hausdorff dimensions,
which can be interesting in their own right. Here we use techniques coming from the theory of topological dimension.

We show that the supremum is actually attained
on the left hand side of the above equation.

We characterize those
compact metric spaces $K$ for which $\dim_{H} f^{-1}(y)=\dim_{t^nH}K-n$ for the generic $f\in C_n(K)$ and the
generic $y\in f(K)$. We also
generalize a result of Kirchheim by showing that if $K$ is
self-similar and $\dim_t K\geq n$ then $\dim_{H} f^{-1}(y)=\dim_{t^nH}K-n$ for the generic
$f\in C_n(K)$ for every $y\in \inter f(K)$.
\end{abstract}

\maketitle

\section{Introduction}

The Hausdorff dimension of a metric
space $X$ is denoted by $\dim_{H} X$, see e.g.~\cite{F} or
\cite{Ma}. In this paper we adopt the convention that
$\dim_{H}\emptyset = -1$.

\bigskip

The following theorem is due to Kirchheim \cite{BK}.

\begin{theorem}[Kirchheim] \label{t:Kirchheim} Let $m,n\in \mathbb{N}^+$, $m\geq n$. For the generic continuous function
$f\colon [0,1]^m\to \mathbb{R}^n$ (in the sense of Baire category) for all $y\in \inter f\left([0,1]^m\right)$
$$\dim_H f^{-1}(y)=m-n.$$
\end{theorem}

Buczolich, Elekes and the author introduced in \cite{BBE}
a new dimension for metric spaces, the topological Hausdorff dimension. The main motivation behind this concept was
to generalize Kirchheim's theorem for real-valued functions
defined on arbitrary compact metric spaces.
We recall first the definition of the (small inductive) topological dimension.

\begin{definition} Set $\dim _{t} \emptyset = -1$. The \emph{topological dimension}
of a non-empty metric space $X$ is defined by induction as
$$
\dim_{t} X=\inf\{d: X \textrm{ has a basis } \mathcal{U}  \textrm{ such that }
\dim_{t} \partial {U} \leq d-1 \textrm{ for every } U\in \mathcal{U} \}.
$$
\end{definition}

For more information on this concept see \cite{E} or \cite{HW}.
The topological Hausdorff dimension (introduced in \cite{BBE}) is defined analogously to the topological
dimension. However, it is not inductive, and it can attain
non-integer values as well.

\begin{definition}\label{deftoph}
Set $\dim _{tH} \emptyset=-1$. The \emph{topological Hausdorff
dimension} of a non-empty metric space $X$ is defined as
\[
\dim_{tH} X=\inf\{d:
 X \textrm{ has a basis } \mathcal{U} \textrm{ such that } \dim_{H} \partial {U} \leq d-1 \textrm{ for every }
 U\in \mathcal{U} \}.
\]
\end{definition}

(Both notions of dimension can attain the value $\infty$ as well, actually we
use the convention $\infty - 1 = \infty$, hence $d = \infty$ is a member of the
above set.)

\bigskip

For more information on topological Hausdorff dimension see \cite{BBE},
here we mention only the results concerning level sets of generic continuous functions.

Let $K$ be a compact metric space, $n$ be a positive integer and let $C_n(K)$ denote the space
of continuous functions from $K$ to $\mathbb{R}^n$ equipped with the supremum norm.
Since this is a complete metric space, we can use Baire category
arguments.

If $\dim_{t}K=0$ then the generic $f\in C_1(K)$ is
well-known to be one-to-one, so every non-empty
level set is a singleton.

Assume $\dim_{t}K>0$. The following theorem from \cite{BBE} shows the
connection between the topological Hausdorff dimension and the level
sets of the generic $f\in C_1(K)$.

\begin{theorem}[Balka, Buczolich, Elekes] \label{t:old}
If $K$ is a compact metric space with $\dim_{t}K>0$ then for the
generic $f\in C_1(K)$
\begin{enumerate}[(i)]
\item $\dim_{H} f^{-1} (y)\leq \dim_{tH} K-1$ for every $y\in \mathbb{R}$,
\item for every $d<\dim_{tH} K$ there exists a non-degenerate interval $I_{f,d}$
such that $\dim_{H} f^{-1} (y)\geq d- 1$ for every $y\in I_{f,d}$.
\end{enumerate}
\end{theorem}

\begin{corollary} \label{c:sup}
If $K$ is a compact metric space with $\dim_t K > 0$ then for the
generic $f \in C_1(K)$
$$\sup_{y\in \mathbb{R}} \dim_{H}f^{-1}(y) = \dim_{tH} K - 1.$$
\end{corollary}

The following definition is due to Darji and Elekes \cite{DE}.

\begin{definition} Let $\dim_{t^nH} \emptyset=-1$ for all $n\in \mathbb{N}$.
For a non-empty metric space $X$ set $\dim_{t^0H}X=\dim_H X$.
The \emph{$n$th inductive topological Hausdorff dimension} is defined inductively as
$$\dim_{t^nH} X=\inf \left\{ d: X \textrm{ has a basis } \mathcal{U} \textrm{ s. t. }
\dim_{t^{n-1}H} \partial {U} \leq d-1 \textrm{ for every } U\in \mathcal{U} \right\}.$$
\end{definition}

The main goal of this paper is to generalize Theorem \ref{t:old} to higher dimensions,
which can be viewed as an application of the inductive topological Hausdorff dimensions.

If $\dim_t K<n$ then the fibers of the generic map $f\in C_n(K)$ are finite, see Theorem
\ref{t:Hurewicz} below.

Suppose $\dim_t K\geq n$. The main theorem of the paper is the following.

\begin{x}[Main Theorem, simplified version] Let $n\in \mathbb{N}^+$ and assume that $K$ is a compact metric space with $\dim_{t}
K\geq n$. Then for the generic $f\in C_n(K)$
\begin{enumerate}[(i)]
\item $\dim_{H} f^{-1} (y)\leq \dim_{t^nH} K-n$ for all $y\in \mathbb{R}^n$,
\item for every $d<\dim_{t^nH} K$ there exists a non-empty
open ball $U_{f,d}\subseteq \mathbb{R}^n$ such that $\dim_{H} f^{-1} (y)\geq
d- n$ for every $y\in U_{f,d}$.
\end{enumerate}
\end{x}

\begin{y} If $K$ is a compact metric space with $\dim_{t}
K\geq n$ then for the generic $f\in C_n(K)$
$$\sup_{y\in \mathbb{R}^n} \dim_H f^{-1}(y)=\dim_{t^nH}K-n.$$
\end{y}

If $K$ is also sufficiently homogeneous, for example self-similar,
then we can actually say more.

\begin{z}[simplified version] Let $K$ be a self-similar compact metric space
such that $\dim_{t} K\geq n$. Then for the generic $f\in C_n(K)$ for any $y\in \inter f(K)$
$$\dim_{H} f^{-1}(y)=\dim_{t^nH}K-n.$$
\end{z}

In the Preliminaries section we introduce some notation and
definitions.

In Section~\ref{s:prop} we prove some basic properties
of inductive topological Hausdorff dimensions.

In order to prove our Main Theorem in Section~\ref{s:equiv} we give some equivalent definitions for the inductive topological Hausdorff dimensions. These equivalent definitions are more or less analogous to the corresponding equivalent definitions of the topological dimension. Throughout the section
we apply the standard techniques of the theory of topological dimension. Perhaps these results can be interesting in their own right, this is the reason why we work in separable metric spaces instead of compact ones.

In Section~\ref{s:val} we completely describe the possible values of the inductive topological Hausdorff dimensions based on
ideas from \cite{BBE}. This is a supplement for the theory of the dimensions, we will not use this result in the subsequent sections.

In Section~\ref{s:main} we consider two more equivalent definitions for the inductive topological Hausdorff dimensions in compact metric spaces, and we
prove the Main Theorem based on Section~\ref{s:equiv}.

In Section~\ref{s:end} we make the Main Theorem more precise. We show that in Corollary~\ref{c:main} the supremum is attained.
We generalize Kirchheim's theorem for sufficiently homogeneous compact spaces.
The proofs of this section rely heavily on the methods developed in \cite{BBE2},
where the case of real-valued functions is investigated.

\section{Preliminaries}

Let $(X,d)$ be a metric space, and let $A,B\subseteq X$ be arbitrary
sets. We denote by $\cl (A)$, $\inter A$ and $\partial A$ the closure,
interior and boundary of $A$, respectively.
The diameter of $A$ is denoted by $\diam A$. We use
the convention $\diam \emptyset = 0$. The distance of the sets $A$
and $B$ is defined by $\dist (A,B)=\inf \{d(x,y): x\in A, \, y\in
B\}$. Let $B(x,r)=\{y\in X: d(x,y)\leq r\}$, $U(x,r)=\{y\in X:
d(x,y)< r\}$ and $B(A,r)=\{x\in X: \dist(B,\{x\})\leq r\}$.

For two metric spaces $(X,d_{X})$ and $(Y,d_{Y})$ a function
$f\colon X\to Y$ is \emph{Lipschitz} if there exists a constant $C
\in \mathbb{R}$ such that $d_{Y}(f(x_{1}),f(x_{2}))\leq C \cdot
d_{X}(x_{1},x_{2})$ for all $x_{1},x_{2}\in X$. A
function $f\colon X\to Y$ is called \emph{bi-Lipschitz}  if $f$ is a
bijection and both $f$ and $f^{-1}$ are Lipschitz.

For every $s\geq 0$ the \emph{$s$-Hausdorff content} of a metric space $X$ is defined as
$$ \mathcal{H}^{s}_{\infty}(X)=\inf \left\{ \sum_{i=1}^\infty (\diam
U_{i})^{s} : X \subseteq \bigcup_{i=1}^{\infty} U_{i} \right\}.$$
Then the \emph{Hausdorff dimension} of $X$ is
$$\dim_{H} X = \inf\{s \ge 0: \mathcal{H}_{\infty}^{s}(X) =0\}.$$
We adopt the convention that $\dim_{H}\emptyset=-1$ throughout the
paper. It is not difficult to see using the regularity of $\mathcal{H}^{s}_{\infty}$ that
every set is contained in a $G_\delta$ set of the same Hausdorff
dimension. For more information on these concepts see \cite{F} or \cite{Ma}.
The following facts are easy consequences of the definitions.

\begin{fact} \label{f:H^s} If $\mathcal{H}^{s}_{\infty} (X)\leq 1$ then
$\mathcal{H}^{t}_{\infty}(X)\leq \mathcal{H}^{s}_{\infty} (X)$ for all $t\geq s$.
\end{fact}

\begin{fact} \label{f:equiv} For a metric space $X$ and $s\geq 0$ the following statements are equivalent:
\begin{enumerate}[(i)]
\item $\dim_H X\leq s$;
\item $\mathcal{H}_{\infty}^{s+\varepsilon}(X)\leq \varepsilon$ for all $\varepsilon>0$;
\item $\mathcal{H}_{\infty}^{s+1/i}(X)\leq 1/i$ for all $i\in \mathbb{N}^+$.
\end{enumerate}
\end{fact}

Let $X$ be a \emph{complete} metric space. A set is \emph{somewhere
dense} if it is dense in a non-empty open set, and otherwise it is
called \emph{nowhere dense}. We say that $M \subseteq X$ is
\emph{meager} if it is a countable union of nowhere dense sets, and
a set is of \emph{second category} if it is not meager. A set is
called \emph{co-meager} if its complement is meager. By the Baire
Category Theorem a set is co-meager iff it contains a dense $G_\delta$ set. We
say that the \emph{generic} element $x \in X$ has property $\mathcal{P}$ if
$\{x \in X : x \textrm{ has property } \mathcal{P} \}$ is co-meager. The set
$A\subseteq X$ has the \emph{Baire property} if $A=U\Delta M$ where $U$ is open
and $M$ is meager. If a set is of second category in every non-empty open set
and has the Baire property then it is co-meager.

If $X$ is a metric space and $A,B$ are disjoint subsets of $X$
then we say that $L\subseteq X$ is a \emph{partition between
$A$ and $B$} if there are open sets $U$, $V$ such that $A\subseteq U$,
$B\subseteq V$, $U\cap V=\emptyset$ and $L=X\setminus (U\cup V)$.
The following lemma is \cite[1.2.11.~Lemma]{E}.

\begin{lemma} \label{l:sep1} Let $X$ be a metric space and let $Z\subseteq X$ be separable with $\dim_t Z=0$. Then for
every pair $A, B$ of disjoint closed subsets of $X$ there exists a partition $L$ between
$A$ and $B$ such that $L\cap Z=\emptyset$.
\end{lemma}

Let us recall the following decomposition theorem for the topological dimension, see \cite[1.5.7.~Thm.]{E} and \cite[1.5.8.~Thm.]{E}.

\begin{theorem} \label{t:topeq} For a separable metric space $X$ and $n\in \mathbb{N}$ the following statements are equivalent:
\begin{enumerate}[(i)]
\item $\dim_t X\leq n$;
\item $X=Y\cup Z$ such that $\dim_t Y\leq n-1$ and $\dim_t Z\leq 0$;
\item $X=Z_1\cup\dots \cup Z_{n+1}$ such that $\dim_t Z_i\leq 0$ for all $i\in \{1,\dots,n+1\}$.
\end{enumerate}
\end{theorem}

Let $X$ be a metric space and let $\mathcal{A},\mathcal{B}$ be families of subsets of $X$, where repeated copies of any given member are allowed. Let
$\mesh \mathcal{A}=\sup \{\diam A: A\in \mathcal{A}\}$. We say that $\mathcal{A}$ is a \emph{cover} of $X$ if $\bigcup \mathcal{A}=X$,
and $\mathcal{A}$ is \emph{locally finite} if every $x\in X$ has a neighborhood that intersects only finitely many $A\in \mathcal{A}$.
The family $\mathcal{A}$ is \emph{open (closed)} if every $A\in \mathcal{A}$ is open (closed) in $X$. We say that $\mathcal{B}$ is a \emph{refinement} of the cover
$\mathcal{A}$ if $\mathcal{B}$ is a cover of $X$ and for every $B\in \mathcal{B}$ there is an $A\in \mathcal{A}$ such that $B\subseteq A$.
The following theorem claims that every metric space is paracompact. It is due to Stone, see \cite[4.4.1.~Thm.]{E2} for a proof.
\begin{theorem}[Stone] \label{t:stone} Every open cover of a metric space has a locally finite open refinement.
\end{theorem}
We say that $\mathcal{B}=\{B_i\}_{i\in I}$ is a \emph{shrinking} of the cover $\mathcal{A}=\{A_i\}_{i\in I}$ if $\mathcal{B}$ is a cover of $X$ and
$B_i\subseteq A_i$ for all $i\in I$. The following theorem is \cite[1.7.8.~Thm.]{E}.
\begin{theorem} \label{t:shrinking} Every finite open cover of a normal space has a closed shrinking.
\end{theorem}

\section{Basic properties of inductive topological Hausdorff dimensions} \label{s:prop}

This section contains some basic properties of inductive topological
Hausdorff dimensions that will be useful in the following sections.

\begin{fact} \label{f:ttt} If $X$ is a metric space and $n\in \mathbb{N}$ then $\dim_t X\leq \dim_{t^nH} X$.
\end{fact}

\begin{proof} If $n=0$ then $\dim_t X\leq \dim_H X$ by \cite[Thm.~VII~2.]{HW}.
Let $n\geq 1$ and assume by induction that the inequality holds for $n-1$. Thus every basis $\mathcal{U}$ of $X$ satisfies
$\dim_t \partial U\leq \dim_{t^{n-1}H}\partial U$ for all $U\in \mathcal{U}$, so the definitions of topological dimension
and $n$th inductive topological Hausdorff dimension imply $\dim_{t}X \leq \dim_{t^nH}X$.
\end{proof}

\begin{fact} \label{f:1} If $X$ is a metric space and $n\in \mathbb{N}$ then
$$\dim_t X<n \quad \Longrightarrow \quad \dim_{t^nH} X=\dim_t X.$$
\end{fact}

\begin{proof}
For $n=0$ the statement is obvious. Let $n\geq 1$ and assume by induction that the
statement holds for $n-1$. Let $X$ be a metric space such that $\dim_t X<n$. As $\dim_{t}X\leq \dim_{t^nH}X$ by Fact \ref{f:ttt},
it is enough to show $\dim_{t^nH}X\leq \dim_t X$. The definition of the topological dimension yields that
$X$ has a basis $\mathcal{U}$ such that $\dim_t \partial U\leq \dim_t X-1<n-1$ for all
$U\in \mathcal{U}$. Then the inductive hypothesis implies that $\dim_{t^{n-1}H} \partial U=\dim_t \partial U$ for
all $U\in \mathcal{U}$, therefore
$$\dim_{t^nH}X\leq \sup_{U\in \mathcal{U}}\dim_{t^{n-1}H}\partial U+1=\sup_{U\in \mathcal{U}} \dim_{t}\partial U+1\leq \dim_t X.$$
This concludes the proof.
\end{proof}

Now we compare the values of different dimensions, for the next theorem see \cite{BBE}.

\begin{theorem} \label{t:<} For every metric space $X$
$$\dim_{t}X\leq \dim_{tH} X \leq \dim _{H} X.$$
\end{theorem}

\begin{fact} \label{f:0} If $X$ is a metric space and $n\in \mathbb{N}$ then $\dim_{t^{n+1}H} X\leq \dim_{t^{n}H} X$.
\end{fact}

\begin{proof} If $n=0$ then this follows from Theorem \ref{t:<}. Let $n\geq 1$ and assume by induction that $\dim_{t^nH}Y\leq \dim_{t^{n-1}H}Y$ for all metric spaces $Y$. Hence for every basis $\mathcal{U}$ of $X$ we have $\dim_{t^{n}H}\partial U\leq \dim_{t^{n-1}H}\partial U$ for all $U\in \mathcal{U}$.
Thus $\dim_{t^{n+1}H} X\leq \dim_{t^{n}H} X$ easily follows from definition of inductive topological Hausdorff dimensions.
\end{proof}

\begin{theorem} \label{t:<n} If $X$ is a metric space and $n\in \mathbb{N}$ then
$$\dim_{t} X\leq \dim_{t^nH} X\leq \dim_{H} X.$$
\end{theorem}

\begin{proof} If $n=0$ or $n=1$ then we are done by Theorem \ref{t:<}. Let $n>1$ and assume by induction that the inequality holds for $n-1$.
Then Fact \ref{f:ttt}, Fact \ref{f:0} and the inductive hypothesis imply $\dim_t X \leq \dim_{t^nH}X\leq \dim_{t^{n-1}H}X\leq \dim_H X$.
\end{proof}

\begin{corollary}[Extension of the classical dimension] The $n$th inductive topological Hausdorff
dimension of a countable set equals zero, and for open subspaces of $\mathbb{R}^{d}$ and
for smooth $d$-dimensional manifolds this dimension equals $d$.
\end{corollary}

\begin{theorem}[Monotonicity]  If $X\subseteq Y$ are metric
spaces then $\dim_{t^nH} X \leq \dim_{t^nH} Y$ for all $n\in \mathbb{N}$.
\end{theorem}

\begin{proof} If $n=0$ then we are done. Let $n\geq 1$ and assume by induction that monotonicity holds for the $(n-1)$st inductive topological Hausdorff dimension. If $\mathcal{U}$ is a basis in $Y$ then $\mathcal{U}_{X}=\{U\cap X: U\in \mathcal{U}\}$ is
a basis in $X$ such that $\partial _{X} (U\cap X)\subseteq \partial_{Y} U$, thus $\dim_{t^{n-1}H} \partial _{X} (U\cap X)\leq \dim_{t^{n-1}H} \partial_{Y} U$  holds for all $U\in \mathcal{U}$. Therefore $\dim_{t^nH} X \leq \dim_{t^nH} Y$.
\end{proof}

\begin{theorem} Let $X,Y$ be metric spaces and $n\in \mathbb{N}$. If $f\colon X\to Y$ is a
Lipschitz homeomorphism then $\dim_{t^nH}Y\leq \dim_{t^nH}X$.
\end{theorem}

\begin{proof} If $n=0$ then we are done. Let $n\geq 1$ and assume by induction that the statement holds for $n-1$.
Since $f$ is a homeomorphism, if $\mathcal{U}$ is a basis in $X$ then
$\mathcal{V}=\{f(U): U\in \mathcal{U}\}$ is a basis in $Y$, and
 $\partial f(U)=f(\partial U)$ for all $U\in \mathcal{U}$.
As $f|_{\partial U}$ is also a Lipschitz homeomorphism, the inductive hypothesis implies that $\dim_{t^{n-1}H} \partial
V=\dim_{t^{n-1}H}\partial f(U)=\dim_{t^{n-1}H}f(\partial U)\leq \dim_{t^{n-1}H} \partial
U$ for all $V=f(U)\in \mathcal{V}$. Therefore $\dim_{t^nH} Y\leq \dim_{t^nH}X$.
\end{proof}

\begin{corollary}[Bi-Lipschitz invariance] \label{c:bi} Let $X,Y$ be metric spaces and $n\in \mathbb{N}$. If $f\colon X\to Y$ is
bi-Lipschitz then $\dim_{t^nH}X=\dim_{t^nH}Y$.
\end{corollary}

\begin{theorem}[Countable stability for closed sets] \label{t:cs}
Let $X$ be a separable metric space with $X=\bigcup_{i=0}^{\infty} X_{i}$, where $X_{i}$
are closed subsets of $X$. Then for all $n\in \mathbb{N}$
$$\dim_{t^nH} X=\sup_{i\in \mathbb{N}} \dim_{t^nH}X_{i}.$$
\end{theorem}

\begin{proof} If $n=0$ then we are done by the countable stability of the Hausdorff dimension.
Let $n\geq 1$ and assume by induction that the statement holds for $n-1$.

Monotonicity clearly implies $\dim_{t^nH}X \geq \sup_{i\in \mathbb{N}} \dim_{t^nH} X_{i}$.
For the other direction we may assume $\sup_{i\in \mathbb{N}}\dim_{t^nH}X_{i}<\infty$. Let $d>\sup_{i\in \mathbb{N}} \dim_{t^nH} X_{i}$ be arbitrary. Let
$\mathcal{U}_{i}$ be a countable basis of $X_{i}$
such that $\dim_{t^{n-1}H} \partial_{X_i} U\leq d-1$ for all $i\in \mathbb{N}$ and $U\in \mathcal{U}_{i}$.

Let $Y=\bigcup \{\partial_{X_i} U : i\in \mathbb{N}, \, U \in
\mathcal{U}_{i}\}$. The countable stability of the $(n-1)$st inductive topological
Hausdorff dimension for closed sets implies $\dim_{t^{n-1}H} Y \leq d-1$. The definition of the
topological dimension yields $\dim_{t} (X_{i}\setminus Y)=0$ for
all $i\in \mathbb{N}$. Then $X_{i}\setminus Y$ is a closed subspace of the
separable metric space $X\setminus Y$, and $X\setminus
Y=\bigcup_{i\in \mathbb{N}} (X_{i}\setminus Y)$.
The countable stability of the topological dimension zero for closed sets
\cite[1.3.1. Thm.]{E} yields $\dim_{t} (X\setminus Y)=0$.

Let us fix an open set $V\subseteq X$ and a point $x\in V$. As
$X\setminus Y$ is a separable subspace of $X$ with $\dim_{t} (X\setminus Y)=0$,
Lemma \ref{l:sep1} yields that there is a partition $L$ between
$\{x\}$ and $X\setminus V$ with $L\subseteq Y$. Thus there exist disjoint open sets $
U,U'\subseteq X$ such that $x\in U$, $X\setminus V\subseteq U'$ and
$U\cup U'=X\setminus L$. In particular, $x \in U
\subseteq V$. Moreover, $\partial_X U\subseteq L\subseteq Y$, thus $\dim_{t^{n-1}H} \partial_X U\leq
\dim_{t^{n-1}H} Y\leq d-1$. Therefore the definition of the $n$th inductive topological Hausdorff dimension implies
$\dim_{t^nH} X\leq d$. As $d>\sup_{i\in \mathbb{N}} \dim_{t^nH} X_{i}$ was arbitrary, the proof is complete.
\end{proof}

\section{Equivalent definitions of inductive topological Hausdorff dimensions} \label{s:equiv}

The aim of this section is to give some equivalent definitions which play a crucial role in the proof of the Main Theorem.

Let $X$ be a separable metric space and $n\in \mathbb{N}^+$. If $\dim_t X<n$ then Fact \ref{f:1} yields $\dim_{t^nH} X=\dim_t X$, so the $n$th inductive topological Hausdorff dimension is reduced to the well-known topological dimension. Hence from now on we can restrict our attention to the case $\dim_t X\geq n$ .

\begin{notation} If $\mathcal{A}$ is a family of sets and $m\in \mathbb{N}^+$ then let $T_{m}(\mathcal{A})$
denote the set of points covered by at least $m$ members of $\mathcal{A}$.
\end{notation}

For a motivation let us repeat the definition of the topological dimension and recall three of its equivalent definitions,
for some details consult \cite{E}.

\begin{theorem} \label{t:topeq2} If $X$ is a non-empty separable metric space then
\begin{align*} \dim_t X=\min\{&n: X \textrm{ has a basis } \mathcal{U} \textrm{ such that }
\dim_{t} \partial {U} \leq n-1 \textrm{ for every } U\in \mathcal{U}\} \\
=\min \{&n: \exists A\subseteq X \textrm{ such that } \dim_t A\leq 0 \textrm{ and } \dim_t (X\setminus A)\leq n-1\} \\
= \min\{&n: \forall (A_1,B_1),\dots,(A_{n+1},B_{n+1}) \textrm{ pairs of disjoint closed subsets } \\
&\textrm{of } X \ \exists \, \textrm{partitions } L_i \textrm{ between } A_i \textrm{ and } B_i \textrm{ such that } \cap_{i=1}^{n+1} L_i=\emptyset\} \\
=\min\{&n:  \forall \, \textrm{finite open cover } \mathcal{U} \textrm{ of } X\ \exists \, \textrm{a finite open refinement } \mathcal{V}  \textrm{ of } \mathcal{U} \\
&\textrm{such that } T_{n+2}(\mathcal{V})=\emptyset\}.
\end{align*}
\end{theorem}

From now on we assume that $n\in \mathbb{N}^+$ and $X$ is a separable metric space with $\dim_t X\geq n$.

\begin{definition} \label{d:equiv} Let
\begin{align*}
P_{t^nH}=\{&d: X \textrm{ has a basis } \mathcal{U} \textrm{ such that }
\dim_{t^{n-1}H} \partial {U} \leq d-1 \textrm{ for every } U\in \mathcal{U}\}, \\
P_{d^nH}=\{&d\geq n: \exists A\subseteq X \textrm{ such that } \dim_H A\leq d-n \textrm{ and } \dim_t (X\setminus A)\leq n-1\}, \\
P_{p^nH}=\{&d\geq n: \forall (A_1,B_1),\dots,(A_n,B_n) \textrm{ pairs of disjoint closed subsets of } X \\
&\exists \, \textrm{partitions } L_i \textrm{ between } A_i \textrm{ and } B_i \textrm{ such that } \dim_H \left(\cap_{i=1}^{n} L_i \right)\leq d-n\}, \\
P_{c^nH}=\{&d\geq n: \forall \varepsilon>0 \ \forall \, \textrm{finite open cover } \mathcal{U} \textrm{ of } X\ \exists \, \textrm{a finite open refinement }\\
&\mathcal{V}  \textrm{ of } \mathcal{U} \textrm{ such that }\mathcal{H}_{\infty}^{d-n+\varepsilon}(T_{n+1}(\mathcal{V}))\leq \varepsilon\}.
\end{align*}
We assume $\infty\in P_{t^nH},P_{d^nH},P_{p^nH},P_{c^nH}$. In the above notation the letter $H$ refers to Hausdorff, while $t,d,p,c$ come from the first letters of the words topological, decomposition, partition and covering, respectively.
\end{definition}

The goal of this section is to prove the following theorem. This implies that the infimum is attained
in the definition of inductive topological Hausdorff dimensions and also yields three equivalent definitions
similarly to Theorem \ref{t:topeq2}.

\begin{theorem}\label{t:equiv} If $n\in \mathbb{N}^+$ and $X$ is a separable metric space with $\dim_t X\geq n$ then
$$\dim_{t^nH} X=\min P_{t^nH}=\min P_{d^nH}=\min P_{p^nH}=\min P_{c^nH}.$$
\end{theorem}

Theorem \ref{t:equiv} easily follows from Lemma \ref{l:inf=min} and Theorem \ref{t:equiv2} below.

\begin{lemma} \label{l:inf=min} $\inf P_{d^nH}\in P_{d^nH}$.
\end{lemma}

\begin{proof} Let $d=\inf P_{d^nH}$, we may assume $d<\infty$. Set $d_i=d+1/i$
for all $i\in \mathbb{N}^+$. As $d_i\in P_{d^nH}$, there exist sets $A_i\subseteq X$
such that $\dim_H A_i\leq d_i-n$ and $\dim_t(X\setminus A_i)\leq n-1$. We may
assume that the sets $A_i$ are $G_{\delta}$, since we can take $G_{\delta}$
hulls with the same Hausdorff dimension. Let $A=\bigcap_{i=1}^{\infty} A_i$,
then clearly $\dim_H A\leq d-n$.  As $X\setminus A_i$ are $F_\sigma$ sets such
that $\dim_t(X\setminus A_i) \leq n-1$ and $X\setminus A\subseteq
\bigcup_{i=1}^{\infty} (X\setminus A_i)$, monotonicity and
countable stability of the topological dimension for $F_{\sigma}$ sets
\cite[1.5.4.~Corollary]{E} yield $\dim_t (X\setminus A)\leq n-1$. Hence $d\in P_{d^nH}$.
\end{proof}

\begin{theorem} \label{t:equiv2} If $n\in \mathbb{N}^+$ and $X$ is a separable metric space with $\dim_t X\geq n$ then
$$P_{t^nH}=P_{d^nH}=P_{p^nH}=P_{c^nH}.$$
\end{theorem}

Before proving Theorem \ref{t:equiv2} we need some preparation.

\begin{notation} If $\mathcal{A}$ is a family of sets and $A\in \mathcal{A}$ then let us define the \emph{star of the set $A$ with respect to the family $\mathcal{A}$} as
$$\St(A,\mathcal{A})=\bigcup \left\{A'\in \mathcal{A}: A\cap A'\neq \emptyset\right\}.$$
\end{notation}

For the next result see \cite[4.1.1.~Lemma]{E} and its proof.

\begin{lemma}\label{l:411} Let $X$ be a normal space and $m\in \mathbb{N}^+$. Assume that $\mathcal{V}_i$ $(i\in \mathbb{N}^+)$ are open covers of $X$ such that $\mathcal{V}_{i+1}$ is a refinement of $\mathcal{V}_i$ for every $i\in \mathbb{N}^+$ and $\left\{\St(V,\mathcal{V}_i): V\in \mathcal{V}_i,\, i\in \mathbb{N}^+\right\}$ is a basis of $X$.
Then every finite open cover $\mathcal{U}$ of $X$ has an open shrinking $\mathcal{V}$ such that $T_{m}(\mathcal{V})\subseteq \bigcup_{i=1}^{\infty} T_{m}(\mathcal{V}_i)$.
\end{lemma}

The following lemma helps us to work with $P_{c^nH}$.

\begin{lemma} \label{l:shr} Let $X$ be a separable metric space and let $m\in \mathbb{N}^+$ and $s\geq 0$.
Then the following properties are equivalent:
\begin{enumerate}[(i)]

\item \label{eq:i} For every $\varepsilon>0$ and open cover $\mathcal{U}$ of $X$ there is an
open refinement $\mathcal{V}$ of $\mathcal{U}$ such that $\mathcal{H}^{s+\varepsilon}_{\infty}(T_{m}(\mathcal{V}))\leq \varepsilon$;

\item \label{eq:ii} for every $\varepsilon>0$ and
finite open cover $\mathcal{U}$ of $X$ there is a
finite open refinement $\mathcal{V}$ of $\mathcal{U}$ such that
$\mathcal{H}^{s+\varepsilon}_{\infty}(T_{m}(\mathcal{V}))\leq \varepsilon$;

\item \label{eq:iii}  for every $\varepsilon>0$ and open cover $\mathcal{U}$ of $X$ there is an open shrinking $\mathcal{V}$ of $\mathcal{U}$ such that
$\mathcal{H}^{s+\varepsilon}_{\infty}(T_{m}(\mathcal{V}))\leq \varepsilon$;

\item  \label{eq:iv} for every $\varepsilon>0$ and
finite open cover $\mathcal{U}$ of $X$ there is an open shrinking $\mathcal{V}$ of $\mathcal{U}$
such that $\mathcal{H}^{s+\varepsilon}_{\infty}(T_{m}(\mathcal{V}))\leq \varepsilon$;

\item  \label{eq:v} for every $\varepsilon>0$ and $m$-element  open cover
$\mathcal{U}=\{U_i\}_{i=1}^{m}$ of $X$ there is an open shrinking $\mathcal{V}=\{V_i\}_{i=1}^{m}$ of $\mathcal{U}$ such that
$\mathcal{H}^{s+\varepsilon}_{\infty}(V_1\cap \dots \cap V_{m})\leq \varepsilon$;

\item  \label{eq:vi} for every open cover $\mathcal{U}_0$ of $X$ there are locally finite open covers
$\mathcal{U}_i$ $(i\in \mathbb{N}^+)$ of $X$ such that for all $i\in \mathbb{N}^+$ we have $\mesh \mathcal{U}_i\leq 1/i$, $\mathcal{H}^{s+1/i}_{\infty}(T_m(\mathcal{U}_i))\leq 1/i$,
and for every $U\in \mathcal{U}_i$ there exists $V\in \mathcal{U}_{i-1}$ such that $\cl (U)\subseteq V$;

\item \label{eq:vii} there exist open covers $\mathcal{U}_i$ $(i\in \mathbb{N}^+)$ of $X$ such that for all $i\in \mathbb{N}^+$ we have $\mesh \mathcal{U}_i\leq 1/i$, $\mathcal{H}^{s+1/i}_{\infty}(T_m(\mathcal{U}_i))\leq 1/i$, and $\mathcal{U}_{i+1}$ is a refinement of $\mathcal{U}_i$.
\end{enumerate}
\end{lemma}

\begin{proof} The proof consists of several implications of different difficulty levels. Proving directions
$(iv) \Rightarrow (i)$ and $(v)\Rightarrow (iv)$ need a lot of effort, while the other directions are more or less obvious.

$(i)\Rightarrow (iii)$ and $(ii)\Rightarrow (iv)$: Let $\mathcal{U}=\{U_i\}_{i\in I}$ be an
open cover of $X$ and let $\mathcal{W}$ be an open refinement of $\mathcal{U}$ such that
$\mathcal{H}^{s+\varepsilon}_{\infty}(T_{m}(\mathcal{W}))\leq \varepsilon$ for some $\varepsilon>0$.
It is enough to find an open shrinking $\mathcal{V}$ of $\mathcal{U}$ with
$\mathcal{H}^{s+\varepsilon}_{\infty}\left(T_{m}(\mathcal{V})\right)\leq \varepsilon$.
For every $W\in \mathcal{W}$ let us choose $i(W)\in I$
such that $W\subseteq U_{i(W)}$. Let $V_i=\bigcup\{W\in \mathcal{W}: i(W)=i\}$ for all $i\in I$
and set $\mathcal{V}=\{V_i\}_{i\in I}$. Then $\mathcal{V}$ is an open shrinking of $\mathcal{U}$ with
$T_{m}(\mathcal{V})\subseteq T_{m}(\mathcal{W})$, so $\mathcal{H}^{s+\varepsilon}_{\infty}\left(T_{m}(\mathcal{V})\right)\leq \varepsilon$.

$(iii)\Rightarrow (ii)$: Straightforward.

$(iv) \Rightarrow (i)$: Let $\mathcal{U}$ be an open cover of $X$ and $\varepsilon>0$, we need to find an open refinement $\mathcal{V}$ of $\mathcal{U}$ with $\mathcal{H}_{\infty}^{s+\varepsilon}(T_m(\mathcal{V}))\leq \varepsilon$. By Fact \ref{f:H^s} we may assume that $\varepsilon<1$.
Suppose that $\mathcal{U}$ is infinite, otherwise we are done. As $X$ is separable, we may assume that $\mathcal{U}$ is countable, let $\mathcal{U}=\{U_j: j\in \mathbb{N}\}$. By Theorem \ref{t:stone} we may suppose that $\mathcal{U}$ is locally finite. Let us enumerate the finite subsets of $\mathbb{N}$ as $\mathbb{N}^{<\omega}=\{Q_i: i\in \mathbb{N}^{+}\}$. For every $i\in \mathbb{N}^{+}$ let us define a closed set by
$$F_i=\bigcap_{j\in Q_i} \cl (U_j) \cap \bigcap_{j\notin Q_i} (X\setminus U_j).$$
Let $\mathcal{V}_0=\{V_{0,j}: j\in \mathbb{N}\}$ be the open cover defined as $V_{0,j}=U_j$ for all $j\in \mathbb{N}$. Assume by induction that $i\in \mathbb{N}^+$ and open covers $\mathcal{V}_{k}=\{V_{k,j}:j\in \mathbb{N}\}$ are already defined for all $k\in \{0,\dots,i-1\}$ such that $\mathcal{V}_{k+1}$ is a shrinking of $\mathcal{V}_k$ for all $k\leq i-2$. Now we define $\mathcal{V}_i$. As $V_{i-1,j}\subseteq V_{0,j}=U_{j}$ for all $j\in \mathbb{N}$, the definition of $F_i$ yields that $V_{i-1,j}\subseteq X\setminus F_i$ if $j\notin Q_i$, so $\{X\setminus F_i, \, V_{i-1,j}:j\in Q_i\}$ is a finite open cover of $X$. Applying property \eqref{eq:iv} for this cover with
$\varepsilon 2^{-i}>0$ and intersecting the open sets with $F_i$ imply that the finite open cover
$\{F_i\cap V_{i-1,j}:j\in Q_i\}$ of $F_i$ has an open shrinking $\mathcal{W}_i=\{W_{i,j}:j\in Q_i\}$ such that
\begin{equation} \label{eq:HWi} \mathcal{H}_{\infty}^{s+\varepsilon 2^{-i}}(T_m(\mathcal{W}_i))\leq \varepsilon 2^{-i}.
\end{equation}
Set $\mathcal{V}_i=\{V_{i,j}: j\in \mathbb{N}\}$, where
$$V_{i,j}=
\begin{cases}
(V_{i-1,j}\setminus F_i)\cup W_{i,j} & \textrm{ if } j\in Q_i, \\
V_{i-1,j} & \textrm{ if } j\notin Q_i.
\end{cases}
$$
It is easy to see that $V_{i,j}$ are open sets in $X$ such that
\begin{equation} \label{eq:ivFi} V_{i-1,j}\setminus V_{i,j}\subseteq F_i \quad  (j\in \mathbb{N}).
\end{equation}
The construction and the inductive hypothesis yield $\left(\bigcup \mathcal{V}_i\right)\cap (X\setminus F_i)=\left(\bigcup \mathcal{V}_{i-1}\right)\cap (X\setminus F_i)=X\setminus F_i$ and $F_i\subseteq \bigcup \mathcal{V}_i$, thus $\mathcal{V}_i$ covers $X$. Hence $\mathcal{V}_i$ is an open shrinking of $\mathcal{V}_{i-1}$.

Let us define $\mathcal{V}=\{V_j:j\in \mathbb{N}\}$ as
$$V_{j}=\bigcap_{i=0}^{\infty} V_{i,j}.$$

We show that $V_j$ is open for all $j\in \mathbb{N}^+$. Let us fix $j\in \mathbb{N}^+$ and $x\in \mathcal{V}_j$. The local finiteness of $\mathcal{U}$ yields that there are $N\in \mathbb{N}$ and an open neighborhood $U$ of $x$ such that $U\cap \cl (U_j)=\emptyset$ for every $j>N$. There exists an $M\in \mathbb{N}$ such that  $Q_i\nsubseteq \{0,\dots,N\}$ if $i>M$, thus $U\cap F_i=\emptyset$ for all $i>M$. As $\mathcal{V}_M$ is a cover, there exists $j\in \mathbb{N}$ such that $x\in V_{M,j}$. Then \eqref{eq:ivFi} yields that $U\cap V_{M,j}\subseteq V_{i,j}$ for all $i>M$, so $U\cap V_{M,j}\subseteq V_{j}$ is an open neighborhood of $x$.

We prove that $\mathcal{V}$ is a cover.
Let $x\in X$ be arbitrarily fixed, the local finiteness of $\mathcal{U}$ yields that there is an $N\in \mathbb{N}$ such that $x\notin \cl (U_j)$ for every $j>N$.
There exists an $M\in \mathbb{N}$ such that $Q_i\nsubseteq \{0,\dots,N\}$ for all $i>M$, thus $x\notin F_i$ for all $i>M$. As $\mathcal{V}_M$ is a cover, there exists $j\in \mathbb{N}$ such that $x\in V_{M,j}$. Then \eqref{eq:ivFi} implies that $x\in V_{i,j}$ for all $i>M$, thus $x\in V_j$.

Now we show that
\begin{equation} \label{eq:tmiv} T_m(\mathcal{V})\subseteq \bigcup_{i=1}^{\infty} T_m(\mathcal{W}_i).
\end{equation}
Assume $x\in T_m(\mathcal{V})$, then there are distinct indexes $j_1,\dots,j_m\in \mathbb{N}$ such that
$x\in V_{j_1}\cap\dots \cap V_{j_m}$. The local finiteness of $\mathcal{U}$ implies that there is a $k\in \mathbb{N}^+$ such that $x\in F_k$. Then $x\in V_{j}\subseteq V_{k,j}$ and $x\notin V_{k-1,j}\setminus F_k$ yields $x\in W_{k,j}$ for all $j\in \{j_1,\dots,j_m\}$, thus $x\in T_m(\mathcal{W}_k)$, so \eqref{eq:tmiv} holds.

Hence $\mathcal{V}$ is an open shrinking (specially a refinement) of $\mathcal{U}=\mathcal{V}_0$. Then \eqref{eq:tmiv},
the subadditivity of $\mathcal{H}_{\infty}^{s+\varepsilon}$, Fact \ref{f:H^s}, and \eqref{eq:HWi} yield
$$\mathcal{H}_{\infty}^{s+\varepsilon}(T_{m} (\mathcal{V}))\leq \sum_{i=1}^{\infty} \mathcal{H}_{\infty}^{s+\varepsilon}(T_{m}(\mathcal{W}_i))
\leq \sum_{i=1}^{\infty} \mathcal{H}_{\infty}^{s+\varepsilon 2^{-i}}(T_m(\mathcal{W}_i))
\leq \sum_{i=1}^{\infty} \varepsilon 2^{-i}=\varepsilon.
$$
Thus property \eqref{eq:i} holds.

$(iv) \Rightarrow (v)$: Straightforward.

$(v)\Rightarrow (iv)$: Let $\mathcal{U}=\{U_i\}_{i=1}^{k}$ be a finite open cover of $X$ and let $\varepsilon>0$.
We need to prove that there exists an open shrinking $\mathcal{V}$ of $\mathcal{U}$ such that $\mathcal{H}^{s+\varepsilon}_{\infty}(T_{m}(\mathcal{V}))\leq \varepsilon$.
We may suppose $k\geq m$, otherwise we are done. By Fact \ref{f:H^s} we may assume that $\varepsilon<1$.

First we prove that there is an open shrinking $\mathcal{W}=\{W_i\}_{i=1}^{k}$ of $\mathcal{U}$ such
that $\mathcal{H}^{s+\varepsilon}_{\infty}(W_1\cap \dots \cap W_{m})\leq \delta$,
where $\delta=\varepsilon/\binom{k}{m}$.
Let us define $\mathcal{U}'=\{U'_i\}_{i=1}^{m}$ such that $U'_i=U_i$ if $i\in \{1,\dots,m-1\}$ and
$U'_m=\bigcup_{i=m}^{k} U_i$.
Then \eqref{eq:v} yields that there is an open shrinking $\mathcal{W}'=\{W'_i\}_{i=1}^{m}$ of $\mathcal{U}'$ such that
\begin{equation} \label{eq:delta} \mathcal{H}^{s+\delta}_{\infty}\left(W'_1\cap \dots \cap W'_{m}\right)\leq \delta.
\end{equation}
Let us define $\mathcal{W}=\{W_i\}_{i=1}^{k}$ such that $W_i=W'_i$ if $i\in \{1,\dots,m-1\}$ and
$W_i=W'_m\cap U_i$ if $i\in \{m,\dots,k\}$. Then $W'_m\subseteq U'_m$ yields $\bigcup_{i=1}^{k} W_i=\bigcup_{i=1}^{m} W'_i=X$, so
$\mathcal{W}$ is an open shrinking of $\mathcal{U}$. Fact \ref{f:H^s} with $\delta<\varepsilon<1$, the definition of $\mathcal{W}$ and \eqref{eq:delta} imply
\begin{align*}
\mathcal{H}^{s+\varepsilon}_{\infty}\left(W_1\cap \dots \cap W_{m}\right) &\leq \mathcal{H}^{s+\delta}_{\infty}(W_1\cap \dots \cap W_{m}) \notag \\
&\leq \mathcal{H}^{s+\delta}_{\infty}\left(W'_1\cap \dots \cap W'_{m}\right)\leq \delta.
\end{align*}

Now the iteration of the above statement yields the required $\mathcal{V}$.
More precisely, let $\mathcal{N}$ be the collection of $m$-element subsets of
$\{1,\dots,k\}$ and let $n=\binom{k}{m}$. Consider a bijection $\phi\colon \{1,\dots,n\}\to \mathcal{N}$.
For $j\in \{1,\dots,n\}$ and $l\in \{1,\dots,m\}$ let $\phi(j,l)$ be the $l$th element of $\phi(j)$
corresponding to the natural ordering. Let $\mathcal{U}_0=\mathcal{U}$ and denote $\mathcal{U}_0=\{U_{0,i}\}_{i=1}^{k}$, where
$U_{0,i}=U_i$. Assume by induction that the open cover
$\mathcal{U}_{j-1}=\{U_{j-1,i}\}_{i=1}^{k}$ of $X$ is
already defined for some $j\in \{1,\dots ,n\}$. Applying the above statement for a rearranged copy of
$\mathcal{U}_{j-1}$ we obtain that there is an open shrinking
$\mathcal{U}_{j}=\{U_{j,i}\}_{i=1}^{k}$ of $\mathcal{U}_{j-1}$ such that
\begin{equation} \label{eq:phij} \mathcal{H}_{\infty}^{s+\varepsilon}\left(\bigcap_{l=1}^{m} U_{j,\phi(j,l)}\right)\leq \delta. \end{equation}
Now $k$-element open covers $\mathcal{U}_0,\dots,\mathcal{U}_n$ of $X$ are defined such that $\mathcal{U}_n$ is a shrinking of $\mathcal{U}_j$ for all $j\in \{0,\dots,n\}$. Therefore the subadditivity of $\mathcal{H}^{s+\varepsilon}_{\infty}$ and \eqref{eq:phij} imply
\begin{align*} \mathcal{H}^{s+\varepsilon}_{\infty}(T_{m}(\mathcal{U}_n))&\leq \sum_{j=1}^{n}\mathcal{H}^{s+\varepsilon}_{\infty} \left(\bigcap_{l=1}^{m} U_{n,\phi(j,l)}\right) \\
&\leq \sum_{j=1}^{n}\mathcal{H}^{s+\varepsilon}_{\infty} \left(\bigcap_{l=1}^{m} U_{j,\phi(j,l)}\right)\leq n\delta=\varepsilon.
\end{align*}
Hence $\mathcal{V}=\mathcal{U}_{n}$ is an open shrinking of $\mathcal{U}$ satisfying \eqref{eq:iv}.

$(iii)\Rightarrow (vi)$: Let $\mathcal{U}_0$ be an open cover of $X$, we need to define open covers $\mathcal{U}_i$ $(i\in \mathbb{N}^+)$ satisfying \eqref{eq:vi}.
Assume by induction that $i\in \mathbb{N}^+$ and $\mathcal{U}_{i-1}$ is already defined. Let $\mathcal{V}_i$ be an open refinement of $\mathcal{U}_{i-1}$ such that
$\mesh \mathcal{V}_i\leq 1/i$ and for every $V\in \mathcal{V}_i$ there exists $U\in \mathcal{U}_{i-1}$ such that $\cl (V)\subseteq U$. By Theorem \ref{t:stone} we may assume that $\mathcal{V}_i$ is locally finite. Applying \eqref{eq:iii} for $\mathcal{V}_{i}$ and $\varepsilon=1/i$ yields that there is an open shrinking $\mathcal{U}_i$ of $\mathcal{V}_i$ such that $\mathcal{H}_{\infty}^{s+1/i}(T_m(\mathcal{U}_i))\leq 1/i$. Then the defined covers $\mathcal{U}_i$ clearly satisfy \eqref{eq:vi}.

$(vi)\Rightarrow (vii)$: Straightforward.

$(vii)\Rightarrow (iv)$: Let $\mathcal{U}$ be a finite open cover of $X$ and let $\varepsilon>0$. We need to find an open shrinking $\mathcal{V}$ of $\mathcal{U}$ with $\mathcal{H}_{\infty}^{s+\varepsilon}(T_m(\mathcal{V}))\leq \varepsilon$. By Fact \ref{f:H^s} we may assume that $\varepsilon<1$.  Assume that open covers $\mathcal{U}_i$ of $X$ are given according to \eqref{eq:vii}. Fix $k\in \mathbb{N}^+$ such that $k\geq 1/\varepsilon$ and let $\mathcal{V}_i=\mathcal{U}_{k2^{i}}$ for all $i\in \mathbb{N}^+$. Since $\mesh \mathcal{V}_i\to 0$ as $i\to \infty$, the family $\left\{\St(V,\mathcal{V}_i): V\in \mathcal{V}_i,\, i\in \mathbb{N}^+\right\}$ is a basis of $X$. Then Lemma \ref{l:411} yields that there is an open shrinking $\mathcal{V}$ of $\mathcal{U}$ such that $T_{m} (\mathcal{V})\subseteq  \bigcup_{i=1}^{\infty} T_{m}(\mathcal{V}_i)$. Therefore the subadditivity of $\mathcal{H}_{\infty}^{s+\varepsilon}$, Fact \ref{f:H^s} for $\frac{1}{k2^{i}}\leq \varepsilon<1$ and property \eqref{eq:vii} imply
\begin{align*} \mathcal{H}_{\infty}^{s+\varepsilon}(T_{m} (\mathcal{V}))&\leq \sum_{i=1}^{\infty} \mathcal{H}_{\infty}^{s+\varepsilon}(T_{m}(\mathcal{V}_i)) \\
&\leq \sum_{i=1}^{\infty} \mathcal{H}_{\infty}^{s+1/(k2^{i})}(T_m(\mathcal{U}_{k2^{i}})) \\
&\leq \sum_{i=1}^{\infty} \frac{1}{k2^{i}}=\frac 1k\leq \varepsilon,
\end{align*}
thus property \eqref{eq:iv} holds.
\end{proof}

The following four lemmas conclude the proof of Theorem \ref{t:equiv2}.

\begin{lemma} \label{l:tndn} $P_{t^nH}\subseteq P_{d^nH}$.
\end{lemma}

\begin{proof} Suppose $d\in P_{t^nH}$ and $d<\infty$, we need to prove that $d\in P_{d^nH}$. The proof is induction on $n$.
We prove the base case $n=1$ and the inductive step for $n>1$ simultaneously, in the latter case we assume by induction that
$P_{t^{n-1}H}(Y)\subseteq P_{d^{n-1}H}(Y)$ for all separable metric spaces $Y$.
Fact~\ref{f:ttt} implies $\dim_{t^nH} K\geq \dim_t K\geq n$, thus $d\geq n$.
There exists a countable basis $\mathcal{U}$ of $X$ such that $\dim_{t^{n-1}H}\partial U\leq d-1$ for all $U\in \mathcal{U}$. Let $F=\bigcup_{U\in \mathcal{U}} \partial U$, then clearly $\dim_t (X\setminus F)\leq 0$.
The countable stability of the $(n-1)$st inductive topological Hausdorff dimension for closed sets yields $\dim_{t^{n-1}H}F\leq d-1$. If $n=1$ then $\dim_H F\leq d-1$ and $\dim_t (X\setminus F)\leq 0$ implies $d\in P_{d^1H}$. If $n>1$ then the inductive hypothesis yields $d-1\in P_{t^{n-1}H}(F)\subseteq P_{d^{n-1}H}(F)$, so
there is a set $A\subseteq F$ such that $\dim_H A\leq (d-1)-(n-1)=d-n$ and $\dim_{t}(F\setminus A)\leq n-2$. Since $(X\setminus A)=(F\setminus A)\cup (X\setminus F)$, Theorem \ref{t:topeq} implies $\dim_t (X\setminus A)\leq n-1$. Thus $A$ witnesses $d\in P_{d^nH}$.
\end{proof}

\begin{lemma} $P_{d^nH} \subseteq P_{p^nH}$.
\end{lemma}

\begin{proof} Assume $d\in P_{d^nH}$ and $d<\infty$. Then there exists $Y\subseteq X$ such that $\dim_H Y\leq d-n$ and $\dim_t (X\setminus Y)\leq n-1$.
Therefore Theorem \ref{t:topeq} yields that $ X\setminus Y=\bigcup_{i=1}^{n}Z_i$ such that $\dim_t Z_i\leq 0$ for all $i\in \{1,\dots,n\}$. Let $(A_1,B_1),\dots,(A_n,B_n)$ be pairs of disjoint closed subsets of $X$. Applying Lemma \ref{l:sep1} for $A_i$, $B_i$ and $Z_i$ yields that there exist partitions $L_i$ between $A_i$ and $B_i$ such that $L_i\cap Z_i=\emptyset$ for all $i\in \{1,\dots,n\}$. Then $\bigcap_{i=1}^{n} L_i\subseteq X\setminus \left(\bigcup_{i=1}^{n}Z_i\right)=Y$, so $\dim_H \left(\bigcap_{i=1}^{n} L_i\right)\leq \dim_H Y\leq d-n$. Thus $d\in  P_{p^nH}$.
\end{proof}

\begin{lemma} \label{l:sncn} $P_{p^nH}\subseteq P_{c^nH}$.
\end{lemma}

\begin{proof} Assume $d\in P_{p^nH}$ and $d<\infty$, we need to prove $d\in P_{c^nH}$. Let
$\varepsilon>0$ and $\mathcal{U}=\{U_i\}_{i=1}^{n+1}$ be an $(n+1)$-element open cover of $X$.
By Lemma \ref{l:shr} it is enough to find an open shrinking $\mathcal{V}=\{V_i\}_{i=1}^{n+1}$ of $\mathcal{U}$
such that $\mathcal{H}_{\infty}^{d-n+\varepsilon}(V_1\cap \dots \cap V_{n+1})\leq \varepsilon$. By Theorem \ref{t:shrinking}
the finite open cover $\mathcal{U}$
has a closed shrinking $\mathcal{A}=\{A_i\}_{i=1}^{n+1}$. For all $i\in \{1,\dots ,n\}$ let us define $B_i=X\setminus U_i$.
The sequence $(A_1,B_1),\dots,(A_n,B_n)$ consists of $n$ pairs of disjoint closed subsets of $X$. Thus $d\in P_{p^nH}$
yields that there exist partitions $L_i$ between $A_i$ and $B_i$ such that $\dim_H \left(\bigcap_{i=1}^n L_i\right) \leq d-n$.
For all $i\in \{1,\dots,n\}$ consider open sets $V_i,W_i\subseteq X$ such that
\begin{equation} \label{eq:VW} A_i\subseteq V_i,\quad B_i\subseteq W_i,\quad V_i\cap W_i=\emptyset,\quad
\textrm{and} \quad X\setminus L_i=V_i\cup W_i.
\end{equation}
Let $L=\bigcap_{i=1}^n L_i$. Then $\dim_H L\leq d-n$ yields $\mathcal{H}_{\infty}^{d-n+\varepsilon}(L)=0$,
so we can choose an open set $U\subseteq X$ such that
$L\subseteq U$ and $\mathcal{H}_{\infty}^{d-n+\varepsilon}(U)\leq \varepsilon$. Let us consider
$\mathcal{V}=\{V_i\}_{i=1}^{n+1}$, where
\begin{equation} \label{eq:VW2}
V_{n+1}=U_{n+1}\cap \left(U\cup \bigcup_{i=1}^{n} W_i\right).
\end{equation}
Now we show that $\mathcal{V}$ is an open shrinking of $\mathcal{U}$.
Since $V_{n+1}\subseteq U_{n+1}$ and \eqref{eq:VW} yields $V_i\subseteq X\setminus B_i=U_i$ for all $i\in\{1,\dots,n\}$,
we need to prove that $\mathcal{V}$ covers $X$.
Equation \eqref{eq:VW} and $A_{n+1}\subseteq U_{n+1}$ imply
\begin{equation} \label{eq:VW3}
\left(\bigcup_{i=1}^{n} V_i\right) \cup U_{n+1}\supseteq \bigcup_{i=1}^{n+1} A_i=X.
\end{equation}
Equation \eqref{eq:VW} and $L\subseteq U$ yield
\begin{align} \label{eq:VW4}
\left(\bigcup_{i=1}^{n} V_i\right) \cup \left(U\cup \bigcup_{i=1}^{n} W_i\right)
=& \bigcup_{i=1}^{n} (V_i\cup W_i)\cup U \notag \\
=&\bigcup_{i=1}^{n}(X\setminus L_i)\cup U \\
=&(X\setminus L)\cup U=X. \notag
\end{align}
Now \eqref{eq:VW2}, \eqref{eq:VW3} and \eqref{eq:VW4} imply $\bigcup_{i=1}^{n+1}V_i=X$, so $\mathcal{V}$ is an open shrinking of $\mathcal{U}$.

Finally, applying \eqref{eq:VW2} and $V_i\cap W_i=\emptyset$ for $i\in \{1,\dots,n\}$ yield
\begin{align*} \label{eq:VW4}
\bigcap_{i=1}^{n+1} V_i\subseteq& \bigcap_{i=1}^{n} V_i \cap \left(U\cup \bigcup_{i=1}^{n} W_i\right) \\
\subseteq& \left(\bigcap_{i=1}^{n} V_i \cap U\right)\cup \left( \bigcap_{i=1}^{n} V_i \cap \bigcup_{i=1}^{n} W_i\right) \\
\subseteq& U\cup \emptyset=U.
\end{align*}
Therefore
$$\mathcal{H}^{d-n+\varepsilon}_{\infty} \left(\bigcap_{i=1}^{n+1} V_i\right)\leq \mathcal{H}^{d-n+\varepsilon}_{\infty}(U)\leq \varepsilon,$$
and the proof is complete.
\end{proof}

\begin{lemma} \label{l:cntn} $P_{c^nH} \subseteq P_{t^nH}$.
\end{lemma}

\begin{proof} Suppose $d\in P_{c^nH}$ and $d<\infty$, we need to prove that $d\in P_{t^nH}$. The proof is induction on $n$.
We prove the base case $n=1$ and the inductive step for $n>1$ simultaneously, in the latter case we assume by induction that
$P_{c^{n-1}H}(Y)\subseteq P_{t^{n-1}H}(Y)$ for all separable metric spaces $Y$.

Let $x\in X$ and let $V_0$ be an open set such that $x\in V_0$. For $d\in P_{t^nH}$ it is enough to construct an open set $U\subseteq V_0$ such that
$x\in U$ and $\dim_{t^{n-1}H} \partial U \leq d-1$. Let $A_0=\{x\}$ and $B_0=X\setminus V_0$. Let $\mathcal{U}_0$ be an open cover of $X$ such that if  $U\in \mathcal{U}_0$ and $\cl (U)\cap A_0\neq \emptyset$ then $\cl (U)\cap B_0=\emptyset$.
Applying Lemma \ref{l:shr} for $s=d-n$ and $m=n+1$ yields that there exist locally finite open covers $\mathcal{U}_i$ $(i\in \mathbb{N}^{+})$ of $X$ such that for all $i\in \mathbb{N}^{+}$ we have $\mesh \mathcal{U}_i\leq 1/i$ and $\mathcal{H}_{\infty}^{d-n+1/i} (T_{n+1}(\mathcal{U}_i))\leq 1/i$, and for every $U\in \mathcal{U}_i$ there exists $V\in \mathcal{U}_{i-1}$ such that $\cl (U)\subseteq V$.

Assume by induction that $i\in \mathbb{N}^{+}$ and disjoint closed sets $A_{i-1}$ and $B_{i-1}$ are already defined.
Consider $A_i=X\setminus G_i$ and $B_i=X\setminus H_i$, where
\begin{align*}
G_i=&\bigcup \left\{U\in \mathcal{U}_i: \cl (U)\cap A_{i-1}=\emptyset\right\}, \\
H_i=&\bigcup \left\{U\in \mathcal{U}_i: \cl (U)\cap A_{i-1}\neq \emptyset\right\}.
\end{align*}
As $G_i$ and $H_i$ are open sets such that $G_i\cup H_i=X$, we obtain that $A_i$ and $B_i$ are disjoint closed sets.

We prove that for all $i\in \mathbb{N}^+$
\begin{equation} \label{eq:123}
\textrm{if } U\in \mathcal{U}_i \textrm{ and } \cl (U)\cap A_{i-1}\neq \emptyset \textrm{ then } \cl (U)\cap B_{i-1}=\emptyset.
\end{equation}
If $i=1$ then the definition of $\mathcal{U}_0$ and the fact that $\mathcal{U}_1$ is a refinement of $\mathcal{U}_0$ imply \eqref{eq:123}.
If $i>1$ and $\cl (U)\cap A_{i-1}\neq \emptyset$ then there is a $V\in \mathcal{U}_{i-1}$ such that $\cl(U)\subseteq V$.
Then clearly $V\cap A_{i-1}\neq \emptyset$, thus $V\nsubseteq G_{i-1}$. Therefore $V\subseteq H_{i-1}$, thus $V\cap B_{i-1}=\emptyset$.
Hence $\cl (U)\cap B_{i-1}=\emptyset$, so \eqref{eq:123} holds.

The local finiteness of $\mathcal{U}_i$ implies that
\begin{align*}
\cl (G_i)=&\bigcup \left\{\cl (U): U\in \mathcal{U}_i \textrm{ and } \cl (U)\cap A_{i-1}=\emptyset\right\}, \\
\cl (H_i)=&\bigcup \left\{\cl (U): U\in \mathcal{U}_i \textrm{ and } \cl (U)\cap A_{i-1}\neq \emptyset\right\}.
\end{align*}
Therefore $\cl (G_i)\cap A_{i-1}=\emptyset$ and \eqref{eq:123} implies $\cl(H_i)\cap B_{i-1}=\emptyset$. Thus we obtain
$A_{i-1}\subseteq X\setminus \cl(G_i)=\inter A_i$ and $B_{i-1}\subseteq X\setminus \cl(H_i)=\inter B_i$. Therefore $U_A=\bigcup_{i=0}^{\infty} A_i$ and $U_B=\bigcup_{i=0}^{\infty} B_i$ are disjoint open sets containing $A_0$ and $B_0$, respectively.

Let us define $L=X\setminus (U_A \cup U_B)=\bigcap_{i=1}^{\infty} (G_i\cap H_i)$,
then $L$ is a partition between $A_0$ and $B_0$. It is enough to prove $\dim_{t^{n-1}H} L \leq d-1$,
because then $x\in U_A\subseteq V_0$, and $\partial U_A\subseteq L$ implies $\dim_{t^{n-1}H} \partial U_A \leq \dim_{t^{n-1}H} L \leq d-1$.
For all $i\in \mathbb{N}^{+}$ consider
$$\mathcal{W}_i=\left\{U\cap L: U\in \mathcal{U}_i \textrm{ and } \cl(U)\cap A_{i-1}\neq \emptyset\right\}.$$

Let us fix $i\in \mathbb{N}^+$. Since $L\subseteq H_i$, we obtain that $\mathcal{W}_i$ is an open cover of $L$, and $\mesh \mathcal{U}_i\leq 1/i$ yields $\mesh \mathcal{W}_i\leq 1/i$.

As $L\subseteq G_i$, for every $x\in L$ there exists an $U\in \mathcal{U}_i$ with $\cl (U)\cap A_{i-1}=\emptyset$, thus $T_{n} (\mathcal{W}_i)\subseteq T_{n+1}(\mathcal{U}_i)$. Hence $\mathcal{H}_{\infty}^{d-n+1/i} (T_{n}(\mathcal{W}_i))\leq 1/i$.

We show that $\mathcal{W}_{i+1}$ is a refinement of $\mathcal{W}_i$. Let $U\cap L\in \mathcal{W}_{i+1}$, where $U\in \mathcal{U}_{i+1}$ and
$\cl(U)\cap A_i\neq \emptyset$. Then there exists $V\in \mathcal{U}_i$ such that $\cl(U)\subseteq V$. Then clearly $V\cap A_i\neq \emptyset$, thus
$V\nsubseteq G_i$, so $\cl(V)\cap A_{i-1}\neq \emptyset$. Hence $V\cap L\in \mathcal{W}_i$ and it contains $U\cap L$.

If $n=1$ then $\mathcal{H}_{\infty}^{d-1+1/i} (L)\leq \mathcal{H}_{\infty}^{d-1+1/i} \left(\bigcup \mathcal{W}_i\right)\leq 1/i$ for all $i\in \mathbb{N}^{+}$, therefore
Fact~\ref{f:equiv} yields $\dim_{t^0H} L=\dim_H L\leq d-1$, and we are done.
If $n>1$ then applying Lemma \ref{l:shr} with $s=d-n$ and $m=n$ for the open covers $\mathcal{W}_i$ implies
$d-1\in P_{c^{n-1}H}(L)$. Then the inductive hypothesis yields $d-1\in P_{t^{n-1}H} (L)$, thus $\dim_{t^{n-1}H} L\leq d-1$. The proof is complete.
\end{proof}

\section{The possible values of inductive topological Hausdorff dimensions} \label{s:val}

In this section we provide a complete
description of the possible values of the inductive topological Hausdorff dimensions. We prove that all values satisfying the conditions of
Facts~\ref{f:ttt}, \ref{f:1}, and \ref{f:0} can be realized even by compact metric spaces. This implies that these dimensions are new and independent
in the following sense: The $n$th inductive topological Hausdorff dimension is not the function of the topological dimension and the $k$th inductive topological Hausdorff dimensions, where $k$ runs over $\mathbb{N}\setminus \{n\}$. This generalizes a theorem in \cite{BBE} concerning the possible values of  the topological Hausdorff dimension. The material developed here will be not used in the subsequent sections.

First we need some preparation. By product of two metric spaces $(X,d_X)$ and $(Y,d_Y)$ we mean the $l^2$-product, that is,
$$d_{X\times Y}((x_{1},y_{1}),(x_{2},y_{2}))= \sqrt{d^{2}_{X}(x_1,x_2)+d^{2}_{Y}(y_1,y_2)}.$$
Now we recall a well-known statement, see \cite[Chapters~3]{F} and \cite[Product formula~7.3]{F} for
the definition of the upper box dimension and the proof, respectively. In fact, \cite{F} works
in Euclidean spaces only, but the proof goes through verbatim to general metric spaces.

\begin{lemma} \label{l:box}
Let $X,Y$ be non-empty metric spaces and let us denote by $\overline{\dim}_{B}$ the upper box dimension. Then
$$\dim_{H}(X\times Y) \leq \dim_{H}X+\overline{\dim}_{B} Y.$$
\end{lemma}

For the sake of notational simplicity we adopt the convention that $[0,1]^0=\{0\}$. Let us recall that $\dim_{t^0H}X=\dim_H X$.

\begin{lemma} \label{l:prod} Let $X$ be a non-empty separable metric space and let $n\in \mathbb{N}$. Then
$$\dim_{t^nH}\left(X \times [0,1]^n\right)=\dim_{H} \left(X\times [0,1]^n\right)=\dim_{H} X+n.$$
\end{lemma}

\begin{proof}
From Theorem \ref{t:<n} it follows that $\dim_{t^nH}\left(X \times [0,1]^n\right)\leq \dim_{H} \left(X\times [0,1]^n\right)$.

Applying  Lemma \ref{l:box} for $Y=[0,1]^n$ we deduce that
$$\dim_{H} \left(X\times [0,1]^n\right) \leq \dim_{H} X+\overline{\dim}_{B} [0,1]^n=\dim_{H} X+n.$$
Finally, we prove that $\dim_H X+n\leq \dim_{t^nH} \left(X\times [0,1]^n\right)$. Let us define
$$\pr_X\colon X\times [0,1]^n\to X, \quad \pr_X (x,y)=x.$$
Let $Z=X\times [0,1]^n$. As $\dim_t Z\geq n$, Theorem~\ref{t:equiv} (moreover, the easy Lemma~\ref{l:tndn})
yields that there is a set $A\subseteq Z$ such that $\dim_H A\leq \dim_{t^nH}Z-n$ and $\dim_t (Z\setminus A)\leq n-1$.
Then $\dim_t (Z\setminus A)\leq n-1<n=\dim_t [0,1]^n$ implies that $A$ intersects $\{x\}\times [0,1]^n$ for all $x\in X$, thus $\pr_X (A)=X$.
Projections do not increase the Hausdorff dimension, thus
$$\dim_{t^nH}Z-n\geq \dim_H A\geq \dim_H \pr_X (A)=\dim_H X.$$
Hence $\dim_{t^nH} \left(X\times [0,1]^n\right) \geq \dim_H X+n$, and the proof is complete.
\end{proof}

Applying the above lemma for $X\times [0,1]^{k-n}$ in place of $X$ yields the following.

\begin{corollary} \label{c:prod} Let $X$ be a non-empty separable metric space and let $n,k\in \mathbb{N}$ with $n\leq k$. Then
$$\dim_{t^nH} \left(X\times [0,1]^k\right)=\dim_H X+k.$$
\end{corollary}

Let $X$ be a non-empty metric space.
If $\dim_t X=\infty$ then Fact~\ref{f:ttt} implies that $\dim_{t^kH} X=\infty$ for all $k\in \mathbb{N}$, thus we may assume that
$\dim_t X=n$ for some $n\in \mathbb{N}$. Fact~\ref{f:1} yields $\dim_{t^kH} X=n$ for all $k>n$, therefore it is enough to describe the possible
$(n+1)$-tuples $(\dim_{t^nH} X,\dots, \dim_{t^0H} X)$. Facts~\ref{f:ttt} and \ref{f:0} imply that
\begin{equation*} \label{eq:dims} n\leq \dim_{t^nH} X\leq \dots \leq \dim_{t^0H} X.
\end{equation*}

The following theorem claims that the above inequality is the only constraint.

\begin{theorem} Let $n\in \mathbb{N}$ and let $d_0,\dots,d_n\in [n,\infty]$ such that $d_{n}\leq \dots \leq d_0$. Then there exists a compact
metric space $K$ such that $\dim_t K=n$ and $\dim_{t^kH} K=d_k$ for all $k\in \{0,\dots,n\}$.
\end{theorem}

\begin{proof} For all $i\in \{0,\dots,n\}$ let $K_i$ be a compact metric space with $\dim_t K_i=0$ and $\dim_H K_i=d_i-i$. It is well-known that there exist such Cantor spaces: If $d_i<\infty$ then $K_i$ can be constructed in a Euclidean space, if $d_i=\infty$ then let $K_i=\prod_{m=1}^{\infty} \frac Cm$ endowed with the $l_2$-metric, where $C$ is the classical `middle-third' Cantor set. Let
$$K=\bigcup_{i=0}^{n} \left(K_i\times [0,1]^i\right),$$
where the union is understood as the disjoint sum of metric spaces.

Let $i\in \{0,\dots,n\}$. Then the Cartesian product theorem \cite[1.5.16.]{E} implies $\dim_t \left(K_i\times [0,1]^i\right) \leq \dim_t K_i+\dim_t [0,1]^i=i$.
By monotonicity we obtain that $\dim_t \left(K_i\times [0,1]^i\right) \geq \dim_t [0,1]^i=i$, thus $\dim_t \left(K_i\times [0,1]^i\right)=i$.
The countable stability of the topological dimension for closed sets \cite[1.5.3.]{E} yields that
$$\dim_t K=\max_{i\in \{0,\dots,n\}} \dim_t \left(K_i\times [0,1]^{i}\right)=\max_{i\in \{0,\dots,n\}} i=n.$$

Let $i,k\in \{0,\dots,n\}$. If $i<k$, then $\dim_{t} \left(K_i\times [0,1]^i\right)=i<k$, therefore Fact~\ref{f:1} yields that
$\dim_{t^kH} \left(K_i\times [0,1]^i\right)=i$. If $i\geq k$ then Corollary~\ref{c:prod} implies $\dim_{t^kH} \left(K_i\times [0,1]^i\right)=\dim_H K_i+i=d_i$. Therefore the stability of the $k$th inductive topological Hausdorff dimension for closed sets yields that
$$\dim_{t^kH} K=\max_{i\in \{0,\dots,n\}} \dim_{t^kH} \left(K_i\times [0,1]^{i}\right)=\max\{0,\dots,k-1,d_k,\dots,d_n\}=d_k.$$
This completes the proof.
\end{proof}

Let $n\in \mathbb{N}$ be arbitrary. The above theorem yields that there exist compact metric spaces $X,Y$ such that $\dim_t X=\dim_t Y$ and $\dim_{t^kH} X=\dim_{t^kH}Y$ for all $k\in \mathbb{N}\setminus \{n\}$ but $\dim_{t^nH}X\neq \dim_{t^nH}Y$. This implies the following corollary.

\begin{corollary} Assume that $n\in \mathbb{N}$. Then $\dim_{t^nH} X$ cannot be calculated from $\dim_t X$ and $\dim_{t^kH} X$ $(k\in \mathbb{N}\setminus \{n\})$,
even for compact metric spaces.
\end{corollary}

\section{The proof of the Main Theorem} \label{s:main}

The goal of this section is to prove Theorem \ref{t:main} based on Section \ref{s:equiv}.
In order to do so we need two new equivalent definitions for the $n$th inductive
topological Hausdorff dimension in compact metric spaces.

For a compact metric space $K$ and $n\in \mathbb{N}^{+}$ let us denote by
$C_n(K)$ the space of continuous functions from $K$ to $\mathbb{R}^n$
equipped with the supremum norm. Since this is a complete
metric space, we can use Baire category arguments.

Let us first note that the case $\dim_t K <n$ is completed by the following theorem of Hurewicz,
see \cite[p. 124.]{Ku}.

\begin{theorem}[Hurewicz] \label{t:Hurewicz}
If $K$ is a compact metric space with $\dim_t K<n$ then $\#f^{-1}(y)\leq n$
for the generic $f\in C_n(K)$ for all $y\in \mathbb{R}^n$.
\end{theorem}

\begin{corollary}
If $K$ is a  compact metric space with $\dim_t K < n$ then every
non-empty fiber of the generic $f\in  C_n(K)$ is of Hausdorff
dimension $0$.
\end{corollary}

Hence from now on we assume that $n\in \mathbb{N}^+$ and a compact metric space $K$ are given such that $\dim_t K\geq n$.

\begin{definition} \label{d:diml} Let
\begin{align*} P_{l^n}=\{&d\geq n: \exists G\subseteq K \textrm{ such that } \dim_H G\leq d-n \textrm{ and we have}  \\
&\# (f^{-1}(y)\setminus G)\leq n \textrm{ for the generic }  f\in C_n(K) \textrm{ for all } y\in \mathbb{R}^n\}.
\end{align*}
\end{definition}

\begin{definition} \label{d:dimw} We say that $f\in C_n(K)$ is \emph{$d$-level
narrow}, if there exists a dense set $S_{f}\subseteq \mathbb{R}^{n}$
such that $\dim_{H} f^{-1}(y)\leq d-n$ for every $y \in S_f$. Let
$\mathcal{N}_n(d)$ be the set of $d$-level narrow functions. Define
$$P_{w^n}=\left\{d\geq n: \mathcal{N}_n(d) \textrm{ is somewhere dense in }C_n(K)\right\}.$$
We assume $\infty\in P_{l^n},P_{w^n}$. The characters $l$ and $w$ come from the first and last letters of the words
level set and narrow, respectively.
\end{definition}

For the definitions of $P_{t^nH}=P_{t^nH}(K)$, $P_{d^nH}=P_{d^nH}(K)$ and $P_{p^nH}=P_{p^nH}(K)$
see Definition \ref{d:equiv} again. Now we show the following theorem.

\begin{theorem} \label{t:5=} If $K$ is a compact metric space with $\dim_{t} K\geq n$ then
$$P_{t^nH}=P_{l^n}=P_{w^n}.$$
\end{theorem}

Theorem \ref{t:5=} and Theorem \ref{t:equiv} immediately yield two new
equivalent definitions for the $n$th inductive topological Hausdorff dimension.

\begin{theorem} \label{t:2=} If $K$ is a compact metric space with $\dim_{t} K\geq n$ then
$$\dim_{t^nH} K=\min P_{l^n}=\min P_{w^n}$$
\end{theorem}

Before proving Theorem \ref{t:5=} we need the following well-known lemma.

\begin{lemma} \label{l:cm}
Let $K_1\subseteq K_2$ be compact metric spaces and
$$R \colon C_n(K_2)\rightarrow C_n(K_1), \quad R(f)=f|_{K_1}.$$
If $\mathcal{F}\subseteq C_n(K_1)$ is co-meager then so is $R^{-1}(\mathcal{F})\subseteq C_n(K_2)$.
\end{lemma}

\begin{proof}
The map $R$ is clearly continuous. Using the Tietze Extension
Theorem it is not difficult to see  that it is also open. We may
assume that $\mathcal{F}$ is a dense $G_{\delta}$ set in $C_n(K_1)$.
The continuity of $R$ implies that $R^{-1}(\mathcal{F})$ is also
$G_{\delta}$, thus it is enough to prove that $R^{-1}(\mathcal{F})$
is dense in $C_n(K_2)$. Let $\mathcal{U} \subseteq C_n(K_2)$ be
non-empty open, then $R(\mathcal{U}) \subseteq C_n(K_1)$ is also
non-empty open, hence $R(\mathcal{U}) \cap \mathcal{F} \neq
\emptyset$, and therefore $\mathcal{U} \cap R^{-1}(\mathcal{F}) \neq
\emptyset$.
\end{proof}

The following four lemmas clearly conclude the proof of Theorem \ref{t:5=}.

\begin{lemma} $P_{p^nH}\subseteq P_{t^nH}\subseteq P_{d^nH}$.  \end{lemma}

\begin{proof} Theorem \ref{t:equiv2} yields the statement. \end{proof}

\begin{lemma} $P_{d^nH}\subseteq P_{l^n}$.
\end{lemma}

\begin{proof} Assume $d\in P_{d^nH}$ and $d<\infty$. There is a set $G\subseteq K$ such that $\dim_H G\leq d-n$ and $\dim_t (K\setminus G)\leq n-1$.
By taking a $G_\delta$ hull of the same Hausdorff dimension we can assume that $G$ is $G_\delta$. As $K\setminus G$ is $F_{\sigma}$, we can choose compact sets $K_i$ such that $K\setminus G=\bigcup_{i=1}^{\infty} K_i$ and $K_i\subseteq K_{i+1}$ for all $i\in \mathbb{N}^+$. For all $i\in \mathbb{N}^+$ let
$$\mathcal{F}_i=\{f\in C_n(K_i):  \forall y\in \mathbb{R}^n,~ \#f^{-1}(y)\leq n \},$$
and let $R_i\colon C_n(K)\to C_n(K_i)$ defined as $R_i(f)=f|_{K_i}$.
As $\dim_t K_i\leq \dim_t(K\setminus G)\leq n-1$, Theorem \ref{t:Hurewicz} implies that the sets
$\mathcal{F}_i\subseteq C_n(K_i)$ are co-meager. Lemma \ref{l:cm} yields that $R_i^{-1}(\mathcal{F}_i)\subseteq C_n(K)$ are co-meager, too.
As a countable intersection of co-meager sets $\mathcal{F}=\bigcap_{i=1}^{\infty} R_i^{-1}(\mathcal{F}_i)\subseteq C_n(K)$ is also co-meager. Clearly, every $f\in \mathcal{F}$ satisfies $\#(f^{-1}(y)\cap K_i)\leq n$ for all $y\in \mathbb{R}^n$ and $i\in \mathbb{N}^+$, so $\bigcup_{i=1}^{\infty} K_i=K\setminus G$ and
$K_i\subseteq K_{i+1}$ $(i\in \mathbb{N}^+)$ yield $\# (f^{-1}(y)\setminus G)\leq n$ for all $f\in \mathcal{F}$ and $y\in \mathbb{R}^n$. Hence $d\in P_{l^n}$.
\end{proof}

\begin{lemma} $P_{l^n}\subseteq P_{w^n}$.
\end{lemma}

\begin{proof} Assume $d\in P_{l^n}$ and $d<\infty$.
The definition of $P_{l^n}$ yields that there exists $G\subseteq K$ such that $\dim_H G\leq d-n$ and $\#(f^{-1}(y)\setminus G)\leq n$ for the generic $f\in C_n(K)$ for all $y\in \mathbb{R}^n$. Then $\dim_H G\leq d-n$ and $d\geq n$ yield $\dim_{H}f^{-1}(y)\leq d-n$, so
$\mathcal{N}_n(d)$ is co-meager, thus (everywhere) dense. Hence $d\in P_{w^n}$.
\end{proof}

\begin{lemma} \label{l:wnsn} $P_{w^n}\subseteq P_{p^nH}$.
\end{lemma}

\begin{proof}
Assume $d\in P_{w^n}$ and $d<\infty$. Then we can fix $f\in C_n(K)$ and $\varepsilon>0$
such that $\mathcal{N}_n(d)$ is dense in $B(f,\varepsilon)$. The uniform continuity of $f$ implies that
there is a $\delta>0$ such that if $A\subseteq K$ with $\diam A\leq \delta$ then $\diam f(A)<\varepsilon/n$.

Now we prove that there is a compact set $C\subseteq K$ such that $\diam C\leq \delta$ and $P_{p^nH}(C)=P_{p^nH}(K)$.
Let us write $K$ as a union of finitely many compact sets with diameter at most $\delta$.
The countable stability of the $n$th inductive topological Hausdorff dimension for closed sets
implies that this union has a compact member $C\subseteq K$ such that $\dim_{t^nH}C=\dim_{t^nH}K$,
and $\diam C\leq \delta$ by definition. Theorem \ref{t:equiv} yields that
$\dim_{t^nH} C=\min P_{p^nH}(C)$ and $\dim_{t^nH} K=\min P_{p^nH}(K)$, so $\min P_{p^nH}(C)=\min P_{p^nH}(K)$.
Hence $P_{p^nH}(C)=P_{p^nH}(K)$, because both sets are of the form $[r,\infty]$.

Finally, it is enough to show $d\in P_{p^nH}(C)$. Let $(A_1,B_1),\dots,(A_n,B_n)$ be arbitrary pairs of disjoint closed subsets of $C$.
We need to show that for all $i\in \{1,\dots,n\}$ there are partitions $L_i\subseteq C$ between $A_i$ and $B_i$ such that
$\dim_H \left(\bigcap_{i=1}^{n} L_i\right)\leq d-n$.
Let $f_1, \dots, f_n \in C_1(K)$ be such that $f = (f_1, \dots, f_n)$ and observe
that we may construct for all $i\in \{1,\dots,n\}$ functions $g_i\in C_1(K)$ such that
\begin{enumerate}[(i)]
\item \label{eq:C1} $\max g_i(A_i)< \min g_i(B_i)$;
\item \label{eq:C2} $g_i \in B(f_i, \varepsilon/n)$;
\item \label{eq:C3} The function $g= (g_1, \dots, g_n) \in C_n(K)$ satisfies $g \in \mathcal{N}_n(d)$.
\end{enumerate}
Indeed, as $\diam f_i(C)\leq \diam f(C)<\varepsilon/n$, for every $i\in \{1,\dots,n\}$ we can define $g_i$ first on
$A_i\cup B_i$ and then we can extend it to $K$ by the Tietze Extension Theorem such that
\eqref{eq:C1} and \eqref{eq:C2} hold. Property \eqref{eq:C2} implies that $g=(g_1,\dots,g_n)\in B(f,\varepsilon)$.
As $g\in B(f,\varepsilon)$ and $\mathcal{N}_n(d)$ is dense in $B(f,\varepsilon)$, we may assume that $g\in \mathcal{N}_n(d)$, so \eqref{eq:C3}
holds.

As $g\in \mathcal{N}_n(d)$, there is a dense set $S_g\subset \mathbb{R}^n$ such that $\dim_H g^{-1}(s)\leq d-n$ for all
$s\in S_g$, see Definition \ref{d:dimw}. We can choose $s=(s_1,\dots,s_n)\in S_g$ such that for every $i\in \{1,\dots,n\}$
its $i$th coordinate $s_i$ satisfies
$\max g_i(A_i)<s_i < \min g_i(B_i)$. Let us define for all $i\in \{1,\dots,n\}$
$$S_i=\{(y_1,\dots,y_n) \in g(K): y_i=s_i\},$$
then \eqref{eq:C1} implies that $S_i$ is a partition between $g(A_i)$ and $g(B_i)$ in $g(K)$ for every
$i\in \{1,\dots,n\}$. For all $i\in \{1,\dots,n\}$ let us define $L_i=(g|_{C})^{-1}(S_i)$.
Then $L_i$ is a partition between $A_i$ and $B_i$ in $C$ such that
$$\bigcap_{i=1}^{n}L_i=\bigcap_{i=1}^{n}(g|_{C})^{-1}(S_i)=
(g|_{C})^{-1}\left( \bigcap_{i=1}^{n} S_i \right)=(g|_{C})^{-1}(s)\subseteq g^{-1}(s).$$
Therefore $s\in S_g$ implies
$$\dim_H \left(\bigcap_{i=1}^{n}L_i\right)\leq \dim_H g^{-1}(s)\leq d-n,$$
thus $d\in P_{p^nH}(C)$. The proof is complete.
\end{proof}

Now we are ready to prove the Main Theorem.

\begin{theorem}[Main Theorem] \label{t:main} Let $n\in \mathbb{N}^+$ and assume that $K$ is a compact metric space with $\dim_{t}
K\geq n$. Then there exists a $G_\delta$ set $G\subseteq K$ with $\dim_H G=\dim_{t^nH} K-n$ such that for the generic $f\in C_n(K)$
\begin{enumerate}[(i)]
\item $\#(f^{-1}(y)\setminus G)\leq n$ for all $y\in \mathbb{R}^n$, thus $\dim_{H} f^{-1} (y)\leq \dim_{t^nH} K-n$ for all $y\in \mathbb{R}^n$,
\item for every $d<\dim_{t^nH} K$ there exists a non-empty open ball $U_{f,d}\subseteq \mathbb{R}^n$ such that $\dim_{H} f^{-1} (y)\geq d- n$ for every
$y\in U_{f,d}$.
\end{enumerate}
\end{theorem}

\begin{proof} Theorem \ref{t:2=} implies that $\dim_{t^nH}K=\min P_{l^n}$, so there exists a set $G\subseteq K$ with
$\dim_H G=\dim_{t^nH} K-n$ such that $\#(f^{-1}(y)\setminus G)\leq n$ for the generic $f\in C_n(K)$ for all $y\in \mathbb{R}^n$.
By taking a $G_\delta$ hull of the same Hausdorff dimension we can assume that $G$ is $G_\delta$.
Then $\dim_{t^nH}K=\min P_{l^n}\geq n$ yields $\dim_H G=\dim_{t^nH} K-n\geq 0$, thus
$\dim_H f^{-1}(y)\leq \dim_H G=\dim_{t^nH} K-n$ for the generic $f\in C_n(K)$ for all $y\in \mathbb{R}^n$.
Hence $(i)$ holds.

Let us now prove $(ii)$. Choose a sequence $d_k \nearrow \dim_{t^nH} K$. Theorem \ref{t:2=}
implies that $d_k <\dim_{t^nH}K=\min P_{w^n}$ for every $k\in \mathbb{N}^{+}$, so
$\mathcal{N}_{n}(d_k)$ is nowhere dense by the
definition of $P_{w^n}$. It follows from the definition of
$\mathcal{N}_n(d)$ that for every $f\in C_n(K)\setminus
\mathcal{N}_{n}(d_k)$ there exists a non-empty open ball
$U_{f,d_k}\subseteq \mathbb{R}^n$ such that $\dim_{H}f^{-1}(y)\geq d_k-n$ for every $y\in U_{f,d_k}$. But then $(ii)$
holds for every $ f \in C_n(K)\setminus \left(\bigcup_{k=1}^{\infty} \mathcal{N}_{n}(d_k)\right)$, and this latter set is
clearly co-meager, which concludes the proof of the theorem.
\end{proof}

\begin{corollary} \label{c:main} If $K$ is a compact metric space with $\dim_{t}
K\geq n$ then for the generic $f\in C_n(K)$
$$\sup_{y\in \mathbb{R}^n} \dim_H f^{-1}(y)=\dim_{t^nH}K-n.$$
\end{corollary}

\section{Strengthening of the Main Theorem} \label{s:end}

The proofs of this section are more or less analogous
to the proofs of the one-dimensional results in \cite{BBE2}, so we only describe the necessary modifications.

\bigskip

Let us fix a compact metric space $K$ and $n\in \mathbb{N}^+$. If $\dim_t K<n$ then the fibers of the generic map $f\in C_n(K)$ are finite,
see Theorem \ref{t:Hurewicz}.

Thus we assume $\dim_t K\geq n$ in the sequel.

\subsection{Fibers of maximal dimension}
Corollary \ref{c:main} states that if $\dim_t K\geq n$ then $\sup_{y\in \mathbb{R}^n} \dim_H f^{-1}(y)=\dim_{t^nH}K-n$
for the generic $f\in C_n(K)$. We show that in this statement the supremum is attained.

\begin{theorem}  \label{t:max} Let $K$ be a compact metric space with $\dim_{t}K\geq n$.
Then for the generic $f\in C_n(K)$
$$\max_{y\in \mathbb{R}^n}\dim_{H}f^{-1}(y)=\dim_{t^nH} K-n.$$
 \end{theorem}

\begin{proof} Buczolich, Elekes and the author proved this theorem for $n=1$, see \cite[Thm. 4.1.]{BBE2}.
The proof goes through with the obvious changes. The only significant
modification is that we need to apply Lemma \ref{l:max} instead of its one-dimensional
special case \cite[Lemma 2.14.]{BBE2}.
\end{proof}

\begin{lemma} \label{l:max} Let $K$ be a compact metric space with a fixed $x_0\in
K$. Let $K_i\subseteq K,$ $i\in \mathbb{N}$ be compact sets such that
\begin{enumerate}[(i)]
\item  $\dim_{t}K_i\geq n$ for all $i\in \mathbb{N}$ and
\item \label{eq:ifi} $\diam \left(K_i\cup \{x_0\}\right)\to 0$ if $i\to \infty$.
\end{enumerate}
Then for the generic $f\in C_n(K)$ we have $f(x_0) \in f(K_i)$ for
infinitely many $i\in \mathbb{N}$.
\end{lemma}

\begin{proof} Clearly it is enough to
show that the sets
$$\mathcal{F}_{k}=\{f\in C_n(K): f(x_0) \notin f(K_i) \textrm{ for all } i\geq k\}$$
are nowhere dense in $C_n(K)$ for all $k\in \mathbb{N}$. Let $k\in \mathbb{N}$, $f_0\in C_n(K)$ and $r>0$ be arbitrarily fixed,
it is enough to find a ball in $B(f_0,3r)\setminus \mathcal{F}_{k}$. We may assume that $x_0\notin K_i$ for all $i\geq k$,
otherwise $\mathcal{F}_k=\emptyset$ and the statement is obvious. By the continuity of $f_0$ and \eqref{eq:ifi} we can fix
$m\geq k$ such that $f_0(K_{m})\subseteq B(f_0(x_0),r)$. As $\dim_t K_{m} \geq n$, by
\cite[Thm. VI 2.]{HW} there is a continuous map $g_0\colon K_{m}\to \mathbb{R}^n$ with a stable value,
that is, there exist $y\in \mathbb{R}^n$ and $\varepsilon>0$ such that $y\in g(K_m)$ for all $g\in B(g_0,2\varepsilon)$.
Therefore $B(y,\varepsilon)\subseteq g(K_m)$ for all $g\in B(g_0,\varepsilon)$.
By applying an affine transformation we may assume that $y=f_0(x_0)$ and $g_0(K_m)\subseteq B(f_0(x_0),r)$, so
$\varepsilon \leq r$. Then $x_0\notin K_m$, $f_0(K_m)\cup g_0(K_m)\subseteq B(f_0(x_0),r)$ and the Tietze Extension Theorem
imply that there is an $f_1 \in B(f_0,2r)$ such that $f_1(x_0)=f_0(x_0)$ and $f_1|_{K_m}=g_0$.
For all $f\in B(f_1,\varepsilon)$ we have $f(x_0)\in B(f_1(x_0),\varepsilon)=B(y,\varepsilon)$
and $f|_{K_m}\in B(g_0,\varepsilon)$ implies $B(y,\varepsilon)\subseteq f(K_m)$. Hence $f(x_0)\in f(K_m)$, so $f\notin \mathcal{F}_k$.
Thus $B(f_1,\varepsilon)\cap \mathcal{F}_k=\emptyset$, so $f_1 \in B(f_0,2r)$ and $\varepsilon\leq r$
imply $B(f_1,\varepsilon)\subseteq B(f_0,3r)\setminus \mathcal{F}_k$.
\end{proof}

\begin{remark} In \cite{BBE2} a compact
set $K\subseteq \mathbb{R}^2$ is constructed such that the generic
$f\in C_1(K)$ has a unique level set of maximal Hausdorff dimension. Therefore we cannot strengthen Theorem
\ref{t:max} in general.
\end{remark}

\subsection{Fibers on fractals}
If $K$ is sufficiently homogenous then
we can improve the Main Theorem.

\begin{definition} \label{suppd}
If $K$ is a compact metric space then let
$$\supp_n  K=\left\{x\in K: \forall r>0,
\  \dim_{t^nH} B(x,r)=\dim_{t^nH} K\right\}.$$
We say that $K$ is \emph{homogeneous for the $n$th inductive topological Hausdorff
dimension} if $\supp_n K=K$.
\end{definition}

\begin{remark} \label{rsupp}
The stability of the $n$th inductive topological Hausdorff dimension for closed sets
clearly yields $\supp_n K \neq \emptyset$. Corollary \ref{c:bi} implies that if $K$ is self-similar then
it is also homogeneous for the $n$th inductive topological Hausdorff dimension.
\end{remark}

\begin{theorem} \label{tip} Let $K$ be a compact metric space with $\dim_{t}K\geq n$. The following
statements are equivalent:
\begin{enumerate}[(i)]
\item  $\dim_{H} f^{-1}(y)=\dim_{t^nH}K-n$ for the generic $f\in C_n(K)$ for generic $y\in f(K)$;
\item  $K$ is homogeneous for the $n$th inductive topological Hausdorff dimension.
\end{enumerate}
\end{theorem}

\begin{proof} The proof of \cite[Thm. 3.3.]{BBE2} works with the obvious modifications. Let us note that the half of this proof is
actually in \cite{BBE}.
\end{proof}

\begin{definition} \label{d:wss} Let $K$ be a compact metric space. We say
that $K$ is \emph{weakly self-similar} if for all $x\in K$ and $r>0$
there exist a compact set $K_{x,r}\subseteq B(x,r)$ and a
bi-Lipschitz map $\phi_{x,r}\colon K_{x,r} \to K$.
\end{definition}

\begin{remark} If $K$ is self-similar then it is also weakly
self-similar. If $K$ is weakly self-similar then Corollary \ref{c:bi} yields that it is
homogeneous for the $n$th inductive topological Hausdorff dimension.
\end{remark}

The following theorem is the main result of the section, it generalizes Kirchheim's theorem for weakly self-similar compact metric spaces.

\begin{theorem} \label{t:Kg} Let $K$ be a weakly self-similar compact metric space
such that $\dim_{t} K\geq n$. Then for the generic $f\in C_n(K)$ for any $y\in \inter f(K)$
$$\dim_{H} f^{-1}(y)=\dim_{t^nH}K-n.$$
\end{theorem}

Before proving Theorem \ref{t:Kg} we need a definition and a lemma.
Basically we follow the proof of \cite[Thm.~3.6.]{BBE2}, but Lemma \ref{l:D} is not similar to \cite[Lemma~2.12.]{BBE2}
and their applications are also different.

\begin{definition} \label{d:D} Let $K$ be a compact metric space and $n\in \mathbb{N}^+$. For all $m\in \mathbb{N}^+$ consider
\begin{align*} \mathcal{D}_m=\{ &f\in C_n(K): \exists \varepsilon>0 \textrm{ such that for all } g\in B(f,\varepsilon) \textrm{ and} \\
&\textrm{for all } y\in g(K)\setminus B(\partial g(K),1/m) \textrm{ we have } y\in f(K)\}.
\end{align*}
If $f\in \mathcal{D}_m$ then one can fix a witness $\varepsilon(f,m)>0$ corresponding to the definition.
\end{definition}

\begin{lemma} \label{l:D} Let $K$ be a compact metric space such that $B(x,r)$ is uncountable
for all $x\in K$ and $r>0$. If $m,n\in \mathbb{N}^+$ then $\mathcal{D}_m=\mathcal{D}_m(K,n)$ is dense in $C_n(K)$.
\end{lemma}

\begin{proof} Let $f_0\in C_n(K)$ and $r>0$ be given, we need to show that $\mathcal{D}_m\cap B(f_0,3r)\neq \emptyset$.
Since $K$ is compact and $f_{0}$ is uniformly continuous, there
are finitely many distinct $x_{1},...,x_{k}\in K$ and $\delta>0$ such that
\begin{equation} \label{eq:K=}  K=\bigcup_{i=1}^{k}B(x_{i},\delta)
\end{equation}
and for all $i\in \{1,\dots,k\}$
\begin{equation} \label{eq:delta2}
f_0(B(x_i,\delta))\subseteq B(f_0(x_i),r/n).
\end{equation}
Choose $0<\delta'<\delta$
such that the balls $B(x_{i},\delta')$ are disjoint. As the balls $B(x_i,\delta'/2)$ are uncountable, for all $i\in \{1,\dots,k\}$ there are sets
$C_i\subseteq B(x_i,\delta'/2)$ homeomorphic to the triadic Cantor set, see \cite[6.5. Corollary]{Ke}.

Let $e_1=(1,0,\dots,0),\dots,e_n=(0,\dots,0,1)$ be the standard basis of $\mathbb{R}^n$. For $y\in \mathbb{R}^n$ and $d>0$ let us denote by $Q(y,d)$
the $n$-cube with center $y$ and edge length $d$ homothetic to $[0,1]^n$. For all $i\in \{1,\dots,k\}$ choose
$d_i\in [2r/n,3r/n]$ such that the $2nk$ many hyperplanes determined by the faces of the cubes $Q_i=Q(f_0(x_i),d_i)$ are distinct.
For all $j\in \{1,\dots,n\}$ let us denote by $\mathcal{S}_j$ the collection of those hyperplanes according to these cubes that are orthogonal to $e_j$.
Set
$$\theta=\min\left\{\dist(S,S'): S,S'\in \mathcal{S}_j,~S\neq S',~1\leq j\leq n\right\}>0.$$

Now we construct $f\in B(f_0,3r)$ such that $f(B(x_i,\delta))=Q_i$ for all $i\in \{1,\dots,k\}$.
As $C_i$ are homeomorphic to the triadic Cantor set, by \cite[4.18.~Thm.]{Ke} there are continuous onto maps $g_i\colon C_i\to Q_i$. Then
$d_i\geq 2r/n$ and \eqref{eq:delta2} imply $f_0(B(x_i,\delta))\subseteq Q_i$, so applying the Tietze
Extension Theorem for the coordinate functions yields that for all $i\in \{1,\dots,k\}$ there are functions $\widehat{g}_i\colon B(x_i,\delta)\to Q_i$
such that $\widehat{g}_i=g_i$ on $C_i$ and $\widehat{g}_i=f_0$ on $B(x_i,\delta)\setminus U(x_i,\delta')$. Let $f(x)=\widehat{g}_i(x)$
for all $x\in B(x_i,\delta)$ and $i\in \{1,\dots,k\}$, then \eqref{eq:K=} implies that $f$ is defined for all $x\in K$. The construction easily
yields that $f\in C_n(K)$ is well-defined and $f(B(x_i,\delta))=Q_i$ for all $i\in \{1,\dots,k\}$, so $f(K)=\bigcup_{i=1}^{k} Q_i$.
Since $f_0(B(x_i,\delta))\cup f(B(x_i,\delta)) \subseteq Q_i$ and $\diam Q_i=\sqrt{n}d_i\leq \sqrt{n}3r/n \leq 3r$,
we obtain that $f\in B(f_0,3r)$.

Finally, we prove that $\varepsilon=\min\{\theta/4,1/(2mn)\}>0$ witnesses $f\in \mathcal{D}_m$. Let $g\in B(f,\varepsilon)$ and
$y_0\in g(K)\setminus B\left(\partial g(K),1/m\right)$, and assume to the contrary that $y_0\notin f(K)$.
We construct points $y_1,\dots,y_n\notin f(K)$ near $y_0$ such that some properties of $y_n$ lead to a contradiction.
Assume by induction that
$y_{j-1}\notin f(K)$ is already defined for some $j\in \{1,\dots,n\}$. The definition of $\theta$ with $2\varepsilon\leq \theta/2$ and $y_{j-1}\notin f(K)=\bigcup_{i=1}^{k} Q_i$ imply that there exists $c_j\in [-2\varepsilon,2\varepsilon]$ such that $y_{j-1}+c_je_j\notin f(K)$ and
\begin{equation} \label{eq:mindist} \min_{S\in \mathcal{S}_j} \dist(\{y_{j-1}+c_je_j\},S) \geq 2\varepsilon.
\end{equation}
Let $y_j=y_{j-1}+c_j e_j$, then $y_1,\dots,y_n$ are defined. The construction yields that $y_n-y_j$ are parallel to $e_j$, so
$\dist(\{y_n\},S)=\dist(\{y_j\},S)$ for all $j\in \{1,\dots,n\}$ and $S\in \mathcal{S}_j$. Therefore
$y_n\notin f(K)=\bigcup_{i=1}^{k} Q_i$ and \eqref{eq:mindist} imply
\begin{align} \label{eq:ydist}
\begin{split}
\dist(\{y_n\},f(K))&\geq \min_{1\leq j\leq n} \min_{S\in \mathcal{S}_j} \dist(\{y_n\},S) \\
&=\min_{1\leq j\leq n} \min_{S\in \mathcal{S}_j} \dist(\{y_j\},S)\geq 2\varepsilon.
\end{split}
\end{align}
Then $y_0\in g(K)\setminus B\left(\partial g(K),1/m\right)$
implies that $B(y_0,1/m)\subseteq g(K)$, and
$$|y_n-y_0|=\sqrt{\sum_{j=1}^{n}c_{j}^2}\leq 2\varepsilon \sqrt{n}\leq 1/m.$$
Hence $y_n\in g(K)$. Choose $x\in K$ such that $g(x)=y_n$, then $g\in B(f,\varepsilon)$ yields
$|f(x)-y_n|=|f(x)-g(x)|\leq \varepsilon$, thus $\dist(\{y_n\},f(K))\leq \varepsilon$, but this contradicts \eqref{eq:ydist}.
Therefore $\mathcal{D}_m\cap B(f_0,3r)\neq \emptyset$, and the proof is complete.
\end{proof}

\begin{proof}[Proof of Theorem \ref{t:Kg}]
The Main Theorem yields $\dim_{H} f^{-1}(y)\leq \dim_{t^nH} K-n$ for the generic $f\in C_n(K)$ for all $y\in \mathbb{R}^n$, thus we only need to verify the
opposite inequality.

Fact \ref{f:ttt} implies $\dim_{t^nH}K\geq \dim_t K\geq n$, therefore we can choose a sequence
$n-1<d_m\nearrow \dim_{t^nH}K$. Let us fix
$m\in \mathbb{N}^{+}$. The Main Theorem implies that for the generic $f\in
C_n(K)$ there exists a non-empty open set $U_{f,d_m}\subseteq \mathbb{R}^n$  such
that $\dim_{H} f^{-1}(y)\geq d_m-n$ for all $y\in U_{f,d_m}$.

By the Baire Category Theorem there are $0<r_1<r_2$ and $y_0\in \mathbb{R}^n$ such that
\begin{align*} \mathcal{H}_m=\{ &f\in C_n(K): f(K)\subseteq B(y_0,r_2) \textrm{ and we have} \\
&\dim_{H} f^{-1}(y)\geq d_m-n \textrm{ for all } y\in B(y_0,r_1)\}
\end{align*}
is of second category. Note that $d_m>n-1$ implies that for every
$f\in \mathcal{H}_m$ we have $B(y_0,r_1)\subseteq f(K)$. Let us define
\begin{align*} \mathcal{G}_{m} = \{&f\in C_n(K): \dim_{H} f^{-1}(y)\geq d_m-n \\
&\textrm{for all } y\in f(K)\setminus B(\partial f(K),1/m)\}.
\end{align*}
Following the proof of \cite[Lemma 3.7.]{BBE2} we obtain that $\mathcal{H}_m$ and $\mathcal{G}_m$ have
the Baire property.

It is sufficient to verify that $\mathcal{G}_m$ is co-meager, since by
taking the intersection of the sets $\mathcal{G}_m$ for all $m\in \mathbb{N}^{+}$
we obtain the desired co-meager set in $C_n(K)$. In order to prove
this we show that $\mathcal{G}_m$ contains `certain copies' of $\mathcal{H}_m$.
Since $\mathcal{G}_m$ has the Baire
property, it is enough to prove that $\mathcal{G}_m$ is of second category
in every non-empty open subset of $C_n(K)$. As $\dim_t K\geq n$, we obtain that $K$ is uncountable,
and the weak self-similarity of $K$ yields that $B(x,r)$ is uncountable for all $x\in K$ and $r>0$.
Hence we can apply Lemma \ref{l:D}. Let us fix an arbitrary
$f_0\in \mathcal{D}_m$ and a witness $\varepsilon=\varepsilon(f_0,m)>0$ corresponding to Definition \ref{d:D}.
As $\mathcal{D}_m$ is dense in $C_n(K)$ by Lemma \ref{l:D}, it is enough to show
that $\mathcal{G}_m \cap B(f_{0},\varepsilon)$ is of second category.

Since $K$ is compact and $f_0$ is uniformly continuous, there are finitely many distinct $x_{1},...,x_{k}\in K$ and
$\delta>0$ such that
\begin{equation}\label{*cover*}
K=\bigcup_{i=1}^{k}B(x_{i},\delta)
\end{equation}
and for each $i\in\{1,\dots,k\}$ the oscillation of $f_{0}$ on
$B(x_{i},\delta)$ is less than
\begin{equation}\label{*a4*}
\omega= \frac{\varepsilon r_1}{2 r_2}< \frac{\varepsilon}{2}.
\end{equation}

Choose $0<\delta'<\delta$ such that the balls $B(x_{i},\delta')$ are disjoint. Using the
weak self-similarity property we can choose for every
$i\in\{1,\dots,k\}$ a set $K_{i} \subseteq B(x_{i},\delta')$ and a
bi-Lipschitz map $\phi_{i}\colon K_i \to K$. Let us fix $i\in
\{1,\dots,k\}$. We define the affine function $\psi_{i}\colon
\mathbb{R}^n\to\mathbb{R}^n$ such that
\begin{equation} \label{psii} \psi_{i}
\left(B(y_0,r_1)\right)=B(f_0(x_i),\omega).
 \end{equation}
Let $G_{i}\colon C_n(K)\rightarrow
C_n(K_{i})$ defined by $G_{i}(f)=\psi_{i}\circ f\circ \phi_{i}$. The maps
$\phi_i\colon K_i \to K$ and $\psi_i \colon \mathbb{R}^n\to \mathbb{R}^n$ are
homeomorphisms, hence $G_i$ is a homeomorphism, too. Therefore, since $\mathcal{H}_{m}$ is of second category in $C_n(K)$, we
obtain that
$${\widehat{{\mathcal F}}_{i}}=\{\psi_{i}\circ f\circ \phi_{i}: f\in \mathcal{H}_m\}=G_{i}(\mathcal{H}_m)$$
is of second category in $C_n(K_{i})$.

Now we prove that ${\widehat{{\mathcal
F}}_{i}}\subseteq B\left(f_{0}|_{K_i},\varepsilon \right)$. Let $f\in \mathcal{H}_m$, then the form of $\psi_i$,
\eqref{psii} and \eqref{*a4*} imply
\begin{align*}
\diam (\psi_{i}\circ f\circ \phi_{i})(K_{i})&=\diam \psi_i(f(K)) \leq
\diam \psi_{i}\left(B(y_0,r_2)\right) \\
&= \frac{r_2}{r_1} \diam \psi_{i}\left(B(y_0,r_1)
\right)=\frac{r_2}{r_1}2\omega=\varepsilon.
\end{align*}
Then $f_0(K_i)\subseteq
B(f_0(x_i),\omega) \subseteq (\psi_{i}\circ f\circ \phi_{i})(K_i)$, so $\psi_{i}\circ f\circ \phi_{i} \in B(f_0|_{K_i},\varepsilon)$. Set
$$\mathcal{F}_{i}=\big\{ f\in B(f_{0},\varepsilon): f|_{K_{i}}\in
{\widehat{{\mathcal F}}_{i}}\big\} \quad \textrm{and} \quad \mathcal{F}=\bigcap_{i=1}^{k}\mathcal{F}_{i}.$$
Clearly $\mathcal{F}\subseteq B(f_{0},\varepsilon)$, and repeating the proof of
\cite[Lemma 3.8.]{BBE2} verbatim yields that
$\mathcal{F}$ is of second category in $B(f_{0},\varepsilon)$.

Therefore it is enough to prove
$\mathcal{F}\subseteq \mathcal{G}_m$, which implies that
$\mathcal{G}_m$ is of second category in $B(f_0,\varepsilon)$. Assume that
$g\in \mathcal{F}$ and $y\in g(K)\setminus B(\partial g(K),1/m)$.
The definition of $\varepsilon=\varepsilon(f_0,m)$ and $g\in B(f_0,\varepsilon)$ yield
$y\in f_{0}(K)$. Hence the definition of $\omega$ and
\eqref{*cover*} imply that there is an $i\in \{1,\dots,k\}$ such
that $y \in B(f_{0}(x_{i}),\omega)$. The
definition of $\mathcal{F}$ yields that there exists an $f\in \mathcal{H}_m$ such
that $g|_{K_{i}}=\psi_{i}\circ f \circ \phi_{i}$.
Then \eqref{psii} implies $\psi_{i}^{-1}(y)\in B(y_0,r_1)$, and
$f\in \mathcal{H}_m$ yields
$\dim_{H}f^{-1}\left(\psi_{i}^{-1}(y)\right)\geq d_m-n$. By  the
bi-Lipschitz property of $\phi _{i}$ we infer
\begin{align*}
\dim_{H} g^{-1}(y)&\geq \dim_{H}(g|_{K_i})^{-1}(y)\\
&=\dim_{H} \phi _{i}^{-1} \left(f^{-1} \left(\psi ^{-1} _{i}
(y)\right)\right) \\
&=\dim_{H} f^{-1} \left(\psi^{-1} _{i} (y)\right)\\
&\geq d_m-n.
\end{align*}
Therefore $g\in \mathcal{G}_m$, so $\mathcal{F}\subseteq \mathcal{G}_m$.
This completes the proof.
\end{proof}

We can analogously define inductive dimensions by replacing Hausdorff dimension with packing, or lower box, or upper box dimension, respectively.
However, one can show that these definitions and some natural modifications of them do not satisfy the analogous version of the Main Theorem. The reason why these concepts behave differently is that box dimensions are not even countable stabile, and packing dimension does not allow to take $G_{\delta}$ hulls: It is easy to see that every $G_{\delta}$ hull of $\mathbb{Q}$ has packing dimension $1$. Therefore the following question is quite natural, the author cannot answer it even in the special case $n=1$. For other problems concerning the topological Hausdorff dimension see \cite{BBE}.

\begin{question}
What is the right notion to describe the packing, or lower box, or upper box dimension of the fibers of the generic continuous function $f\in C_n(K)$?
\end{question}

\subsection*{Acknowledgement} The author is indebted to U.~B.~Darji, M.~Elekes and the anonymous referees for some helpful suggestions.

\end{document}